\documentclass[11pt]{amsart}

\usepackage{amssymb,latexsym}

\usepackage{graphicx,epsfig}

\usepackage[dvipsnames]{xcolor}
\def\cal{\mathcal}

\def\Bbb{\mathbb}

\def\R{{\Bbb R}}	

\def\ext{{\cal E}} % extension operator\cal

\def\cP{{\cal P}}
\def\I{{\cal I}}
\def\J{{\cal J}}
\def\cR{{\cal R}}

\def\Landau{{\cal O}}

 \def\Om{{\Omega}}
\def\ka{{\kappa}}

\def\de{{\delta}}

\def\eps{{\epsilon}}

\def \supp {\text{\rm supp\,}}
\def \graph {\text{\rm graph\,}}

\def\trans{{\,}^t}

\def\pa{{\partial}}

\newtheorem{thmnr}{Theorem}[section]
\newtheorem{lemnr}[thmnr]{Lemma}

\newtheorem{definr}{Definition}[section]

\theoremstyle{nonumberbreak} % ab hier keine nummerierung und nächste zeile nach
\begin{document}

\title[On  Fourier restriction for finite-type  perturbations of the  hyperboloid]{On  Fourier restriction for finite-type  perturbations of the hyperbolic paraboloid%\\
%VERSION  22.7.19
}

\author[S. Buschenhenke]{Stefan Buschenhenke}
\address{S. Buschenhenke:  Mathematisches Seminar, C.A.-Universit\"at Kiel,
Ludewig-Meyn-Stra\ss{}e 4, D-24118 Kiel, Germany}
\email{{\tt buschenhenke@math.uni-kiel.de}}
\urladdr{{http://analysis.math.uni-kiel.de/buschenhenke/}}

\author[D. M\"uller]{Detlef M\"uller}
\address{D. M\"uller: Mathematisches Seminar, C.A.-Universit\"at Kiel,
Ludewig-Meyn-Stra\ss{}e 4, D-24118 Kiel, Germany}
\email{{\tt mueller@math.uni-kiel.de}}
\urladdr{{http://analysis.math.uni-kiel.de/mueller/}}

\author[A. Vargas]{Ana Vargas}
\address{A. Vargas: Departmento de Mathem\'aticas, Universidad Aut\'onoma de  Madrid, 28049 Madrid,
Spain
}
\email{{\tt ana.vargas@uam.es}}
\urladdr{{http://matematicas.uam.es/~AFA/}}

\thanks{2010 {\em Mathematical Subject Classification.}
42B25}
\thanks{{\em Key words and phrases.}
hyperbolic  hypersurface, Fourier restriction}
\thanks{The first author was partially supported by the ERC grant 307617.\\
The first two authors were partially supported by the DFG grant MU 761/
11-2.\\
The third author was partially supported by grants MTM2013--40945 (MINECO) and
MTM2016-76566-P (Ministerio de Ciencia, Innovaci$\acute{\text{o}}$n y Universidades), Spain.}

\begin{abstract}  In this note, we continue our  research on  Fourier restriction for hyperbolic surfaces, by studying local  perturbations
of the hyperbolic paraboloid $z=xy,$ which are  of the form $z=xy+h(y),$ where $h(y)$ is a smooth function of finite type. Our results build on previous joint work in which we have studied the case $h(y)=y^3/3$  by means of the bilinear method. As it turns out, the understanding of that special case  becomes  also crucial for the treatment of arbitrary  finite type perturbation terms $h(y).$

\end{abstract}

\maketitle

%\vfill\newpage

\tableofcontents

\thispagestyle{empty}

\setcounter{equation}{0}
\section{Introduction}\label{intro}

Our aim in this note is to provide another  step in our program towards gaining an understanding
of  Fourier restriction for general hyperbolic surfaces.

Fourier restriction for hypersurfaces  with non-negative principal curvatures has been studied intensively by many authors (see, e.g., \cite{Bo1}, \cite{Bo2}, \cite{Bo3},\cite{MVV1}, \cite{MVV2}, \cite{TV1}, \cite{TV2},  \cite{W2}, \cite{T4}, \cite{ikm}, \cite{lv10}, \cite{IM-uniform}, \cite{BoG}, \cite{IM}, \cite{bmv16}, \cite{Gu16},
\cite{Gu17}, \cite{Sto17a}). For the case of hypersurfaces of  non-vanishing  Gaussian curvature but principal curvatures of different signs, besides Tomas-Stein type Fourier restriction estimates (see, e.g.,  \cite{To},\cite{Str}, \cite{Gr},\cite{St1}, \cite{ikm}, \cite{IM-uniform},  \cite{IM}), until recently  the only case which had been studied successfully was the case of the hyperbolic paraboloid (or "saddle") in  $\R^3$:
in 2015, independently S. Lee \cite{lee05} and A. Vargas \cite{v05}  established results analogous to Tao's theorem \cite{T2} on elliptic surfaces (such as the $2$ -sphere), with the exception of the  end-point, by means of the bilinear method.  Recently, B. Stovall
\cite{Sto17b}  was able to include also the end-point case. Moreover,  C. H. Cho and J. Lee \cite{chl17}, and
J. Kim \cite{k17},  improved the range by adapting ideas by Guth  \cite{Gu16},
\cite{Gu17} which are based on the polynomial partitioning method. For further  information on the history of the restriction problem, we refer the interested reader to our previous paper \cite{bmv17}.

\medskip

We shall here study surfaces $S$ which are  local perturbations of the hyperbolic paraboloid $z=xy,$ which are given as the graph of a function
$\phi(x,y):=xy+h(y),$ where the function $h$ is smooth and of finite type at the origin, i.e.,

\begin{align}\label{surface}
S:=\{(x,y,xy+h(y)):(x,y)\in \Om\},
\end{align}
where $\Om$ is a sufficiently small neighborhood of the origin, and $h(y)=y^{m+2} a(y),$ with $a(0,0)\ne 0$  and $m\ge 1.$
The Fourier restriction problem, introduced by E. M. Stein in the seventies (for general submanifolds), asks for  the range  of exponents
$\tilde p$ and $\tilde q$ for which  an a priori  estimate of the form
\begin{align*}
\bigg(\int_S|\widehat{f}|^{\tilde q}\,d\sigma\bigg)^{1/\tilde q}\le C\|f\|_{L^{\tilde
p}(\R^n)}
\end{align*}
holds  true for   every Schwartz function   $f\in\cal S(\R^3),$ with a constant $C$
independent of $f.$ Here, $d\sigma$ denotes the surface measure on $S.$

As usual, it will be more convenient to use duality and work in the adjoint setting.  If
$\cR$ denotes the Fourier restriction operator $g\mapsto \cR g:=\hat g|_S$ to the surface
$S,$ its adjoint operator $\cR^*$ is given by $\cR^*f(\xi)=\ext f(-\xi),$ where
 $\ext$ denotes the ``Fourier extension'' operator given by
\begin{align*}%\label{defop}
	\ext f(\xi):=\widehat{f\,d\sigma}(\xi)= \int_S f(x)e^{-i\xi\cdot x}\,d\sigma(x),
\end{align*}
with  $f\in L^q(S,\sigma).$ The restriction problem is therefore
equivalent to the question of finding the appropriate range of exponents for which the
estimate
$$
\|\mathcal E f\|_{L^r(\R^3)}\le C\|f\|_{L^q(S,d\sigma)}
$$
holds true with a constant $C$ independent of the function   $f\in L^q(S,d\sigma).$

By identifying a point $(x,y)\in \Om$ with the corresponding point $(x,y,\phi(x,y))$ on $S,$
we may regard our Fourier extension operator $\ext$  as well as an operator mapping
functions on $\Om$ to functions on $\R^3,$ which in  terms of our phase function
$\phi(x,y)=xy+h(y)$ can be expressed  more explicitly  in the form
$$
\ext f(\xi)=\int_{\Om} f(x,y) e^{-i(\xi_1 x+\xi_2 y+\xi_3\phi(x,y))} \eta(x,y) \, dx dy,
$$
if $\xi= (\xi_1,\xi_2,\xi_3)\in \R^3,$ with a suitable  smooth density $\eta.$

Our main result, which generalizes Theorem 1.1 in \cite{bmv17},  is the following

\begin{thmnr}\label{mainresult}
	Assume that $r>10/3$ and  $1/q'>2/r,$ and let $\ext$ denote  the Fourier extension
operator associated to the graph $S$  in \eqref{surface} of the above phase function $\phi(x,y):=xy+h(y),$ where the function  $h$ is smooth and of finite type at the origin. Then, if $\Om$ is a sufficiently small neighborhood of the origin,
	\begin{align*}
		\|\ext f\|_{L^r(\R^3)} \leq C_{r,q} \|f\|_{L^q(\Om)}
	\end{align*}
	for all $f\in L^q(\Om)$.
\end{thmnr}
For the proof of this result, we shall strongly build on the approach devised for the special case where $h(y)=y^3/3.$ In many  arguments, we shall be able to basically follow \cite{bmv17}. Therefore, we shall concentrate on explaining  the new ideas and modifications that are needed to   handle   more general finite type perturbations.

\bigskip

\noindent\textsc{Convention:}
Unless  stated otherwise, $C > 0$ will stand for an absolute constant whose value may vary
from occurrence to occurrence. We will use the notation $A\sim_C B$  to express  that
$\frac{1}{C}A\leq B \leq C A$.  In some contexts where the size of $C$ is irrelevant we
shall drop the index $C$ and simply write $A\sim B.$ Similarly, $A\lesssim B$ will express
the fact that there is a constant $C$ (which does not depend  on the relevant quantities
in
the estimate) such that $A\le C B,$  and we write  $A\ll B,$ if the constant $C$ is
sufficiently small.
\medskip

\begin{center}
{\textsc{Acknowledgments}}
\end{center}
The authors would like to express their sincere gratitude to the referee for  many valuable  suggestions which have greatly helped to improve the presentation of the material in this article.

%%%%%%%%%%%%%%%%%%%%%%%%%%%%%%

\setcounter{equation}{0}
\section{Reduction to perturbations of cubic type}\label{sec:cubred}
Recall that we are assuming that
\begin{equation}\label{}
\phi(x,y)=xy+y^{m+2} a(y), \quad \text{where}\ a(0)\ne0, m\ge 1.
\end{equation}
We may assume without loss of generality that $\Om$  is a square, and then  decompose the domain $\Om$ dyadically with respect to the $y$-variable into rectangular boxes
$$\Om=\bigcup\limits_{\pm}\bigcup\limits_{i\ge i_0} \Om^\pm_{2^{-i}},
$$
 where for any $\ka=2^{-i}$ we have $\Om^{-}_\ka=-\Om^+_{\ka},$
and $\ka\le y\le 2\ka$ on $\Om^+_\ka.$  Note that we may assume that $i_0\gg 1$ is sufficiently large, by  choosing $\Om$ sufficiently small. By
 $$
\ext^{\pm}_{\ka} f(\xi)=\int_{\Om^{\pm}_{\ka}} f(x,y) e^{-i(\xi_1 x+\xi_2 y+\xi_3\phi(x,y))} \eta(x,y) \, dx dy
$$
we denote the contribution of $\Om^{\pm}_{\ka}$ to $\ext f.$
\medskip

Let us fix one of these subsets, say $\Om^+_\ka.$ We then apply an
affine change of variables to pass to the phase
$$\phi_\ka(x,y):= \frac 1\ka \phi\big(x,\ka(1+y)\big)=x(1+y)+ \ka^{m+1}(1+y)^{m+2} a\big(\ka (1+y)\big),
$$
where $0\le y\le 1$. Actually, by taking, say, $1000$ subdomains, we may even assume that $0\le y\le 1/1000.$ Let us put
$$
H_\ka(y):= (1+y)^{m+2} a\big(\ka (1+y)\big).
$$
Then
$$
\phi_\ka(x,y)=x(1+y) +\ka^{m+1} H_\ka(y) =x+xy +\ka^{m+1} P_2(\ka,y)+\ka^{m+1} h_\ka(y),
$$
where $P_2(\ka,y)$ denotes the Taylor polynomial of $H_\ka(y)$ of degree $2$ centered at $y=0.$  As in our previous paper \cite{bmv17},
we may then write
\begin{eqnarray*}
x+xy +\ka^{m+1} P_2(\ka,y)&=xy+c_\ka y^2+ \text{affine linear terms}\\
&=(x+c_\ka y)y+ \text{affine linear terms}.
\end{eqnarray*}
The linear change of variables $x\mapsto x+c_\ka y$ then allows to reduce to the  phase function
$$
\tilde \phi_\ka(x,y):= xy+\ka^{m+1} h_\ka(y),
$$
for $(x,y)$ in a sufficiently small neighborhood of the origin which can be chosen independently of $\ka.$ Note that
\begin{equation}\label{cubic1}
h_\ka(0)=h'_\ka(0)=h''_\ka(0).
\end{equation}
%\begin{eqnarray}
%h''_\ka(y)=H''_\ka(y)&&=(m+2)(m+1) (1+y)^m a\big(\ka (1+y)\big) \nonumber\\
%&&\quad+2\ka(m+2) (1+y)^{m+1}a'\big(\ka (1+y)\big)+\ka^2 (1+y)^{m+2}a''\big(\ka (1+y)\big) .\nonumber
%\end{eqnarray}
Moreover, it is easy to see that  for $\ka$ sufficiently small (depending on $m, \, a(0)\ne 0,\, \|a'\|_\infty,\|a''\|_\infty$ and $\|a'''\|_\infty$), we have
\begin{equation}\label{cubic2}
|h'''_\ka(y)|\ge   \tfrac{(m+2)(m+1)m}2 |a(0)|\ge C_3>0.
\end{equation}
A similar reasoning shows that
\begin{equation}\label{cubic3}
|h'''_\ka(y)|\le 4 C_3, \quad \text{and}\quad  |h^{(l)}_\ka(y)|\le C_l \ \   \text{for all } |y|\le 1/1000, l\ge 4,
\end{equation}
with constants $C_l$ which are independent of $\ka.$

Similar arguments apply to $\Om^{-}_\ka.$ We  consider next the Fourier extension operator
$$
\tilde\ext^{\pm}_{\ka} f(\xi)=\int_{\tilde \Om_\ka^{\pm}} f(x,y) e^{-i(\xi_1 x+\xi_2 y+\xi_3(xy+\ka^{m+1} h_\ka(y))} \tilde \eta_\ka(x,y) \, dx dy,
$$
where $\tilde \eta_\ka(x,y)=\eta(x,\kappa(1+y)),$ which corresponds to the operator $\ext^{\pm}_{\ka}$ in the new coordinates. Then an  easily scaling argument shows  that  the following estimates for $\ext^{\pm}_{\ka}$ and
$\tilde\ext^{\pm}_{\ka}$ are equivalent:
\begin{equation}\label{extka1}
\|\tilde\ext^{\pm}_{\ka} f\|_{L^r}\le C \|f\|_{L^q};
\end{equation}
\begin{equation}\label{extka2}
\|\ext^{\pm}_{\ka} g\|_{L^r}\le C \ka^{1-2/r-1/q} \|g\|_{L^q}
\end{equation}
for all g with $\supp g\subset \{|y-\ka|\le \ka/1000\} $ (and support in $x$ sufficiently small).

Since we work under the assumption that $1/q'>2/r,$  we thus see that by summing a geometric series  it will suffice to prove the uniform estimates \eqref{extka1} in order to prove Theorem \ref{mainresult}.

%%%%%%%%%%%%%%%%%%%%%%%%%%%%%%

\setcounter{equation}{0}
\section{Transversality conditions and admissible pairs of sets}\label{sect:admissible}
In the previous section, we have seen that we may reduce to proving uniform Fourier extension estimates for phases
$$
\phi(x,y)=xy+\epsilon h(y),
$$
defined on a  small square $Q$ which, after a further scaling, we may assume to be the square  $Q=[-1,1]\times [-1,1],$
where $\epsilon>0$ is assumed to be sufficiently small,  and where $h$ is a  perturbation  function {\it  of cubic type} in $y$ of the  phase $xy.$    By this, we mean that  $h$ is smooth and satisfies
\begin{eqnarray}\label{cubtype}
\begin{cases}  & h(0)=h'(0)=h''(0)=0,\\
&\frac {C_3}4\le |h'''(y)|\le C_3 \quad\text{for all}\  |y|\le 1,\\
&|h^{(l)}(y)|\le C_l \quad\text{for all   } l\ge 4  \ \text{and}\   |y|\le 1
\end{cases}
\end{eqnarray}
(compare \eqref{cubic1}--\eqref{cubic3}, where we have applied an additional scaling by a factor 1000 in $y$). Here, the constants $C_l$ will be  assumed to be fixed constants,  with $C_3>0,$ and our goal will be to establish uniform estimates which will depend only on these constants (in many parts actually only on $C_3$), but not on $\epsilon.$

\subsection{Admissible pairs of sets $U_1,$ $U_2$ on which transversalities are of a fixed
size: an informal discussion}\label{pairs of sets}
%Copied:
Recall next that the bilinear approach is based on  bilinear estimates of the
form
\begin{align}\label{bil1}
 \|\ext_{U_1}(f_1)\,\ext_{U_2}(f_2)\|_p \leq C(U_1,U_2) \|f_1\|_2\|f_2\|_2.
\end{align}
 Here,  $\ext_{U_1}$ and $\ext_{U_2}$ are the Fourier extension operators associated to
 patches of sub-surfaces $S_i:=\graph \phi|_{U_i}\subset S,\ i=1,2,$  with $U_i\subset \Om.$
 What is crucial for obtaining  useful bilinear estimates is that the two  patches of
 surface $S_1$ and $S_2$ satisfy certain {\it transversality conditions,} which  are
 stronger than  just assuming that $S_1$ and $S_2$ are transversal as hypersurfaces (i.e.,
 that all normals to $S_1$ are transversal to all normals to $S_2$). Indeed, what  is
 needed in addition is the following (cf. \cite {bmv17},\cite{lee05}, \cite{v05}, \cite{lv10}, or \cite{be16}):
 \smallskip

Denoting by  $H\phi$ the Hessian of $\phi,$ we consider the following quantity
\begin{align}\label{transs}
\tilde\Gamma^\phi_{z}(z_1,z_2,z_1',z_2'):=	\left\langle
(H\phi)^{-1}(z)(\nabla\phi(z_2)-\nabla\phi(z_1)),\nabla\phi(z_2')-\nabla\phi(z_1')\right\rangle.
\end{align}
 If its modulus  is  bounded from below by a constant $c>0$  for all $z_i=(x_i,y_i),\, z_i'=(x_i',y_i')\in U_i$, $i=1,2$, $z=(x,y)\in U_1\cup U_2,$ then we have \eqref{bil1} for $p>5/3,$ with a constant $C(U_1,U_2)$ that depends only on this constant  $c$  and on upper bounds for the derivatives of $\phi.$  If $U_1$ and
$U_2$ are sufficiently small (with sizes depending on upper bounds of the first and second
order derivatives of $\phi$ and a lower bound for the determinant of $H\phi$) this
condition  reduces to the estimate
\begin{align}
|\Gamma^\phi_{z}(z_1,z_2)|\geq c,
\end{align}
for $z_i=(x_i,y_i)\in U_i$, $i=1,2$, $z=(x,y)\in U_1\cup U_2$, where
\begin{align}\label{trans}
\Gamma^\phi_{z}(z_1,z_2):=	\left\langle
(H\phi)^{-1}(z)(\nabla\phi(z_2)-\nabla\phi(z_1)),\nabla\phi(z_2)-\nabla\phi(z_1)\right\rangle.
\end{align}
It is easy to check that for $\phi(x,y)=xy+\epsilon h(y)$, we have
\begin{eqnarray} \label{gammaz}
\Gamma^\phi_{z}(z_1,z_2)
	 &=:& 2(y_2-y_1)\,\tau_{z}(z_1,z_2),
\end{eqnarray}
where
\begin{equation}\label{TV1}
\tau_{z}(z_1,z_2):=x_2-x_1+\epsilon[h'(y_2)-h'(y_1)-\frac 12 h''(y)(y_2-y_1)].
\end{equation}
As in \cite{bmv17}, it will be particularly important to look at the
expression
\eqref{TV1} when $z=z_1\in U_1,$ and $z=z_2\in U_2,$  so that the two ``transversalities''
\begin{align}\label{TV2}
	\tau_{z_1}(z_1,z_2)=x_2-x_1+\epsilon[(h'(y_2)-h'(y_1)-\frac 12 h''(y_1)(y_2-y_1)]\\
	\tau_{z_2}(z_1,z_2)=x_2-x_1+\epsilon[(h'(y_2)-h'(y_1)-\frac 12 h''(y_2)(y_2-y_1)] \label{TV3}
\end{align}
become relevant.
Note  the following relation  between these quantities:
\begin{align}\label{TV2+TV3}
|\tau_{z_1}(z_1,z_2)-\tau_{z_2}(z_1,z_2)|&=\frac \epsilon 2 |h''(y_2)-h''(y_1)||y_2-y_1|\nonumber
\sim \eps  |h'''(\eta)|(y_2-y_1)^2\\
&\sim \eps (y_2-y_1)^2,
\end{align}
where $\eta$ is some intermediate point.
\medskip

Following Section 2 in \cite{bmv17}, we shall try to devise neighborhoods $U_1$ and $U_2$ of two given points
 $z_1^0=(x_1^0,y_1^0)$ and $z_2^0=(x_2^0,y_2^0)$ on which these quantities are roughly
  constant for $z_i=(x_i,y_i)\in U_i,$ $i=1,2$,
 and which  are also  essentially chosen as large as possible. The corresponding pair $(U_1,U_2)$ of
 neighborhoods of $z^0_1$ respectively $z^0_2$ will be called an {\it admissible pair}.

As in \cite{bmv17}, we will present the basic motivating idea in this subsection, and give a precise definition of admissible pairs
 in the next subsection.
\medskip

 In a first step, we choose a large constant $C_0\gg 1$, which will be made precise only later,
 and assume that
 $|y^0_2-y^0_1|\sim C_0\rho$ for some $\rho>0.$  It is then natural to allow $y_1$  to vary on  $U_1$ and $y_2$
 on $U_2$  by at most $\rho$ from $y^0_1$ and
 $y^0_2,$ respectively, i.e., we shall assume that
  \begin{align*}%\label{y-var}
|y_i-y^0_i|\lesssim \rho, \qquad \text{for}\quad  z_i\in U_i,\, i=1,2,
\end{align*}
so that indeed
\begin{equation}\label{rhosize}
|y_2-y_1|\sim C_0\rho \qquad\text{for}\quad  z_i\in U_i,\, i=1,2.
\end{equation}

 Recall next  the  identity \eqref{TV2+TV3}, which in particular implies that
  \begin{align}\label{TV2+TV3'}
|\tau_{z^0_1}(z^0_1,z^0_2)-\tau_{z^0_2}(z^0_1,z^0_2)|\sim C_0^2\eps\rho^2.
\end{align}
We begin with

\medskip

\noindent {\bf Case 1: Assume that $ |\tau_{z_1^0}(z_1^0,z_2^0)|\le |\tau_{z_2^0}(z_1^0,z_2^0)|.$ }
Let us then write
\begin{equation}\label{defdelta}
 |\tau_{z_1^0}(z_1^0,z_2^0)|= \eps\rho^2\de,
\end{equation}
where $\de\ge 0.$  Note, however, that  obviously $\eps\rho^2\de\lesssim1.$  From \eqref{TV2+TV3'} one then easily
deduces that there are two subcases:

\medskip

\noindent {\bf Subcase 1(a): (the ``straight box'' case),} where  $|\tau_{z_1^0}(z_1^0,z_2^0)| \sim
|\tau_{z_2^0}(z_1^0,z_2^0)|,$  or, equivalently, $\de\gtrsim1.$ In this case,  also
$|\tau_{z_2^0}(z_1^0,z_2^0)|\sim
\eps\rho^2\de.$

\medskip

\noindent {\bf Subcase 1(b): (the ``curved box'' case),} where  $|\tau_{z_1^0}(z_1^0,z_2^0)| \ll
|\tau_{z_2^0}(z_1^0,z_2^0)|,$  or, equivalently, $\de\ll 1.$ In this case,  $|\tau_{z_2^0}(z_1^0,z_2^0)|\sim
\eps\rho^2.$

\medskip
Given $\rho$ and $\delta,$ we shall then want to devise $U_1$ and $U_2$ so that the same kind of conditions hold
for all $z_1\in U_1$ and $z_2\in U_2,$
i.e.,
$$
 |\tau_{z_1}(z_1,z_2)|\sim \eps\rho^2\de,\text{  and} \quad  |\tau_{z_2}(z_1,z_2)|\sim \eps\rho^2(1\vee \de).
$$
Note that in view of  \eqref{TV2+TV3} and \eqref{rhosize} the second condition is redundant,
and so the  only additional condition that needs  to be satisfied  is that, for all
$z_1=(x_1,y_1)\in
U_1$ and $z_2=(x_2,y_2)\in U_2,$ we have
\begin{align*}%\label{transUi}
|\tau_{z_1}(z_1,z_2)|\sim \eps\rho^2\de.
\end{align*}

\medskip
The  choice of the sets $U_1$ and $U_2$ becomes particularly lucid if we first assume that $z_1^0=0,$ so let us begin by examining this case.  Later we shall see that a simple change of coordinates will allow to reduce to this case for general $z_1^0.$

\medskip
{\bf The case $z_1^0=0:$}
We  shall  want to choose $U_2$  as large  as possible w.r.  to $y_2,$ so we assume that on $U_2$ we have
$|y_2-y_2^0|\lesssim\rho.$ Let $$a^0:=\tau_{0}(0,z_2^0),$$ so  that $|a^0|\sim
\eps\rho^2\de.$  Then we shall assume that on $U_2$ we have, say,
$|\tau_{0}(0,z_2)-a^0|\ll\eps \rho^2\de.$

If $z_1^0=0,$ this means that  we shall  {\bf define  $U_2$}  by the following conditions:
%Put
%\begin{align}
%	a^0:=\tau_{0}(0,z_2^0)=x_2^0+\eps h'(y_2^0),
%\end{align}
% and assume that $ |a^0|\sim \eps \rho^2 \de.$
% Then we require that
\begin{eqnarray}\label{U2c0}
\begin{split}
|y_2-y_2^0|&\lesssim \rho, \\
\ |\tau_{0}(0,z_2)- \tau_{0}(0,z_2^0)|&=|x_2+\eps h'(y_2)-a^0|\ll \eps \rho^2\de.%\nonumber
\end{split}
\end{eqnarray}

% Note for later use that if we replace the point $z_2^0=(x_2^0,y_2^0)$ here by an in $x$  shifted point $z_2^e=(x_2^e, y_2^0),$ then
%\begin{eqnarray}\label{ztoze}
%\tau_{z_1^0}(z_1^0,z_2)- \tau_{z_1^0}(z_1^0,z^e_2)&=&x_2-x_2^e+\eps[h'(y_2)-h'(y_2^0)-\tfrac{h''(y_1^0)}2 (y_2-y_2^0)]
%&=&x-\big(x_2^ e+\eps[h'(y_2^0)-\tfrac{h''(y_1^0)}2 y_2^0]\big)+\eps[h'(y)-\tfrac{h''(y_1^0)}2 y]\label{ztoze}
%\end{eqnarray}

As for $U_1,$ given our choice of $U_2,$ what we  still need is that $|\tau_{z_1}(z_1,z_2)-\tau_{0}(0,z_2)|\ll \eps\rho^2\de$ for all $z_1\in U_1$ and $z_2\in U_2,$ for then also $|\tau_{z_1}(z_1,z_2)-\tau_{0}(0,z_2^0)|\ll \eps\rho^2\de$ for all such $z_1, z_2.$

Note that, for %$z_1^0=0$ and
$y_2$ fixed,  the equation
\begin{eqnarray}\nonumber
0= \tau_{z_1}(z_1,z_2)-\tau_{0}(0,z_2)=-\Big(x_1 +\eps[h'(y_1)
+\tfrac {h''(y_1)}2(y_2-y_1)]\Big) \label{difftaus}
\end{eqnarray}
defines a curve $x_1=\gamma(y_1),$ so that the condition $|\tau_{z_1}(z_1,z_2)-\tau_{0}(0,z_2)|\ll \eps\rho^2\de$  determines  essentially an $\eps\rho^2\de$ neighborhood of this curve, whose slope $\pa_{y_1}\gamma$ is of order $O(\eps).$ Moreover, since $y_2$ is allowed to vary within $U_2$ of order $O(\rho),$ and since   \eqref{cubtype}  shows that $|\pa_{y_2}(\pa_{y_1}\gamma)|=|\eps h'''(y_1)/2|\sim \eps,$ we see that the natural condition to impose for $U_1$ is that
$\eps \rho|y_1-y_1^0|=\eps \rho|y_1|\ll \eps\rho^2\de,$ i.e.,
$$|y_1|\le  \rho\de\wedge \rho=\rho(1\wedge \de)
$$
 (note here that, in Subcase 1(a), we may have $\de\ge 1$).
   Moreover,  by the mean value theorem and \eqref{cubtype}, we have  $|h'(y_1)|\sim |h'''(\eta)|y_1^2\sim C_3 y_1^2$  and $|h''(y_1)|\sim |h'''(\tilde \eta)y_1|\sim C_3|y_1| $
 whereas $|y_2-y_1|\sim \rho.$ Thus we see  that  $|\eps[h'(y_1)+\tfrac {h''(y_1)}2(y_2-y_1)]|\ll \eps \rho^2 \de.$
 \medskip

  In combination, this shows that it will be natural to {\bf define  $U_1$} by the following  conditions:
 \begin{eqnarray}\label{U1c0}
 \begin{split}
&|y_1|\lesssim \rho(1\wedge \de), \\
&\big|x_1|\ll \eps\rho^2\de.
 \end{split}
\end{eqnarray}

\medskip
{\bf The case of arbitrary $z_1^0:$} Let now  $z_1^0:=(x_1^0,y_1^0)$ be arbitrary.   In a  first step  we translate the point $z_1^0$ to the origin by writing
$ z=z_1^0+\tilde z,$ i.e., $x=x_1^0+\tilde x, y=y_1^0 +\tilde y.$  Then
\begin{eqnarray*}
\phi(z)&=&\phi(z_1^0+\tilde z)=(x_1^0+\tilde x)(y_1^0+\tilde y)+\eps h(y_1^0+\tilde y)\\
&=& \tilde x \tilde y +\eps \tfrac{h''(y_1^0)}2(\tilde y)^2+\eps H(\tilde y) + \text{affine linear terms}\\
&=& \big(\tilde x+\eps \tfrac{h''(y_1^0)}2\tilde y\big)\tilde y+\eps H(\tilde y) + \text{affine linear terms},
\end{eqnarray*}
with
\begin{equation}\label{fdef}
H(\tilde y)=h(\tilde y+y_1^0)- h(y_1^0) -h'(y_1^0)\tilde y -\tfrac {h''(y_1^0)}2 (\tilde y)^2.
\end{equation}
\color{black}
%$H(\tilde y):=h(y_1^0+\tilde y)-P_2(\tilde y),$ where $P_2$ denotes the Taylor polynomial of order $2$  at $y_1^0,$ i.e.,
%$P_2(\tilde y):=h(y_1^0) +h'(y_1^0) \tilde y +\tfrac {h''(y_1^0)}2 (\tilde y)^2.$
 By our assumptions \eqref{cubtype}  on $\phi,$  the error term $H$ satisfies estimates of the  form
\begin{eqnarray}\label{cuberror}
\begin{cases}  & H(0)=H'(0)=H''(0)=0,  \\
&|H'''(\tilde y)|=|h'''(y_1^0+\tilde y)|\sim C_3,\\
&|H^{(l)}(\tilde y)|\le C_l \quad\text{for all}\   l\ge 4,
\end{cases}
\end{eqnarray}
which means that also $H$ is of cubic type, uniformly in $z_1^0,$ with the same constants $C_l$ as for $h.$

It is thus natural to introduce a further change of coordinates
\begin{equation}\label{newcoord}
x'':=\tilde x+\eps  h''(y_1^0)\tilde y=x-x_1^0+\eps \tfrac {h''(y_1^0)}2 (y-y_1^0), \, y'':=\tilde y=y-y_1^0,
\end{equation}
so that in these coordinates

\begin{equation}\label{newcoordphi}
\phi(z)=x'' y''+\eps H(y'') + \text{affine linear terms}.
\end{equation}
This shows that in these coordinates $(x'',y''),$ the function $\phi$ is again a perturbation of $x'' y''$ by a  perturbation function $H(y'')$  of cubic type in the sense of \eqref{cubtype} (up to an affine linear term, which is irrelevant),   uniformly in the parameter $z_1^0.$
\color{black}

We can now define the sets $U_1$ and $U_2$ by choosing them in terms of the coordinates  $(x'',y'')$ as in \eqref{U1c0} and \eqref{U2c0}, only with the function $h$ replaced by $H,$ and then express those sets in terms of our original coordinates $(x,y).$  Note also that in the coordinates $(x'',y''),$ we have
$$
(x_1^0)''=0, (y_1^0)''=0\quad\text{and} \quad  (x_2^0)''=x_2^0-x_1^0+\eps \frac{h''(y_1^0)}2 (y_2^0-y_1^0), (y_2^0)''=y_2^0-y_1^0,
$$
 and $\tau_{z_1^0}(z_1^0,z_2)=x_2''+\epsilon H'(y_2'').$ In combination with \eqref{fdef} this  then leads  to the following {\bf  choices of $U_1$ and $U_2:$ }

 We {\bf define  $U_1$} by the  conditions
 \begin{eqnarray}\label{U1c}
  \begin{split}
|y_1-y_1^0|&\lesssim \rho(1\wedge \de), \\
\big|x_1-x_1^0 +\eps \tfrac {h''(y_1^0)}2(y_1-y_1^0)|&\ll \eps\rho^2\de,
 \end{split}
\end{eqnarray}

and   $U_2$  by the conditions
\begin{eqnarray}\label{U2c}
  \begin{split}
|y_2-y_2^0|&\lesssim \rho, \\
|x_2-x_1^0 +\eps[h'(y_2)-h'(y_1^0) -\tfrac {h''(y_1^0)}2(y_2-y_1^0)]-a^0|&\ll \eps \rho^2\de,
  \end{split}
\end{eqnarray}
where
\begin{align}\label{a0defi}
	a^0:=\tau_{z_1^0}(z_1^0,z_2^0)
\end{align}
  is assumed to be of size   $|a^0|\sim \eps \rho^2 \de.$

\noindent{\bf Note:}  $U_1$ is essentially the affine image of a rectangular box of dimension $\eps\rho^2\de \times
\rho(1\wedge \de).$  However, when $\de\ll 1,$ then $U_2$ is a thin curved box, namely the segment of an
$\eps\rho^2\de$-neighborhood of a curve of curvature $\sim\eps$ lying within the  horizontal strip where $|y_2-y_2^0|\lesssim \rho.$ On
the other hand, when $\de\gtrsim 1,$   then it is easily  seen that  $U_2$ is essentially a rectangular box  of
dimension $\eps\rho^2\de\times \rho.$ This explains  why we  called Subcase 1(b)  where $\de\ll1$  the ``curved box
case'', and   Subcase 1(a)  where $\de\gtrsim 1$  the ``straight  box case.''

\medskip
\noindent {\bf Case 2: Assume that $|\tau_{z_1^0}(z_1^0,z_2^0)|\ge |\tau_{z_2^0}(z_1^0,z_2^0)|.$ }

This case  can easily be reduced to the previous one by symmetry. By \eqref{TV1}, we have $\tau_z(z_1,z_2)=-\tau_z(z_2,z_1)$. Hence we just
need to interchange the roles of $z_1$ and $z_2$ in the previous discussion, so that it is natural here to
define
%an {\it admissible pair} $(\tilde U_1,\tilde U_2)$  {\it of type 2}  of neighborhoods $\tilde U_1$ of $z^0_1$
%respectively $\tilde U_2$ of $z^0_2$ by setting
%
%\begin{eqnarray}\label{tildeU}
%  \begin{split}
%\tilde U_1=\{(x_2,y_2): |y_1-y_1^0|\lesssim \rho,\ |x_1-x^0_2+y_1(y_1-y_2^0)- \tilde a^0|\ll \rho^2\de\},\\
%\tilde U_2=\{(x_1,y_1): |y_2-y_2^0|\lesssim \rho(1\wedge \de), |x_2-x_2^0+y_2^0(y_2-y_2^0)|\ll
%\rho^2\de\},
% \end{split}
%\end{eqnarray}
%where $\tilde a^0:=\tau_{z_2^0}(z_2^0,z_1^0).$
% We {\bf define
$\tilde U_1$ by the  conditions
 \begin{eqnarray}\label{U3c}
  \begin{split}
|y_1-y_1^0|&\lesssim \rho, \\
\big|x_1-x_2^0 +\eps [h'(y_1)-h'(y_2^0)-\tfrac {h''(y_2^0)}{2}(y_1-y_2^0)]-a^0|&\ll \eps\rho^2\de,
 \end{split}
\end{eqnarray}
 where $a^0=\tau_{z_2^0}(z_2^0,z_1^0)=-\tau_{z_2^0}(z_1^0,z_2^0),$  and   $\tilde U_2$  by the conditions
\begin{eqnarray}\label{U4c}
  \begin{split}
|y_2-y_2^0|&\lesssim \rho(1\wedge \de), \\
|x_2-x_2^0 +\eps\tfrac {h''(y_2^0)}2(y_2-y_2^0)|&\ll \eps \rho^2\de.
  \end{split}
\end{eqnarray}
\bigskip

\medskip

\subsection{Precise definition of admissible pairs within $Q\times Q$}\label{preciseadmissible}
In view of our discussion in the previous subsection, we shall here  devise  more precisely
certain \lq\lq dyadic''    subsets   of $Q\times Q$ which will
assume the roles of the sets $U_1,$ respectively $U_2, $ in such a way  that on every pair of such sets each of
our transversality functions  is essentially of some fixed dyadic size, and which will moreover  lead to  a kind
of Whitney decomposition of $Q\times Q$ (as will be shown in  Section \ref{bilinlin}). Again, this mimics the approach in \cite{bmv17}, namely Section 2.2.
To begin with, as before we fix a large dyadic constant $C_0\gg1.$

\smallskip
In  a first step,  we   perform a  classical {\bf dyadic decomposition in the
$y$-variable}
which is a variation of the one in \cite{TVV}:
For a given  dyadic number $0<\rho\lesssim 1,$  we denote for
$j\in\mathbb Z$ such that $|j|\rho\le 1$  by $I_{j,\rho}$ the dyadic interval
$I_{j,\rho}:=[j\rho,j\rho+\rho)$ of length $\rho,$ and by $V_{j,\rho}$ the corresponding
horizontal ``strip''   $V_{j,\rho}:=[-1,1]\times I_{j,\rho}$ within $Q.$
Given two dyadic intervals $J,\,J'$ of the same size, we say that they are
{\it related} if their parents are adjacent but they are not adjacent. We divide each
dyadic interval $J$ in a disjoint union of dyadic subintervals $\{I_J^k\}_{1\le k\le
C_0/8},$ of length  $8|J|/C_0.$  Then, we define $(I,I')$ to be an \it
admissible pair of dyadic intervals \rm if and only if there are $J$ and $J'$ related
dyadic intervals and $1\le k,\,j\le C_0/8$ such that $I=I_J^k$ and $I'=I_{J'}^j.$

We say that a pair of strips $(V_{j_1,\rho},V_{j_2,\rho})$ is {\it admissible } and  write
$V_{j_1,\rho}\backsim V_{j_2,\rho},$ if  $(I_{j_1,\rho},I_{j_2,\rho})$ is a pair of
admissible dyadic intervals. Notice that in this case,
  \begin{align}\label{admissibleV}
C_0/8< |j_2-j_1|< C_0/2.
\end{align}
One can  easily see that  this leads to the following disjoint decomposition of $Q\times
Q:$
\begin{align}\label{whitney1}
Q\times Q= \overset{\cdot}{\bigcup\limits_{\rho}}
\,\Big(\overset{\cdot}{\bigcup\limits_{V_{j_1,\rho}\backsim
V_{j_2,\rho}}}V_{j_1,\rho}\times V_{j_2,\rho}\Big),
\end{align}
where the first union is meant to be over all such dyadic $\rho$'s.

\medskip
In a second step, we perform a non-standard {\bf Whitney type decomposition of any given
admissible pair of strips}, to obtain subregions in which the transversalities are roughly
constant.

To simplify notation, we fix $\rho$ and an admissible pair  $(V_{j_1,\rho},V_{j_2,\rho}),$
and simply write $I_i:=I_{j_i,\rho},\, V_i:=V_{j_i,\rho}, \, i=1,2,$
so that $I_i$ is an interval of length $\rho$ with left endpoint  $j_i\rho,$   and
\begin{align}\label{Vi}
V_1=[-1,1]\times I_1, \qquad V_2=[-1,1]\times I_2,
\end{align}
are rectangles of dimension $2\times \rho,$ which are vertically  separated at scale
$C_0\rho.$
More precisely, for $z_1=(x_1,y_1)\in V_1$ and $z_2=(x_2,y_2)\in V_2$ we have
$|y_2-y_1|\in
|j_2\rho-j_1\rho|+[-\rho,\rho],$ i.e.,
\begin{align}\label{yseparation}
	C_0\rho/2\le |y_2-y_1|\le  C_0\rho.
\end{align}

Let $0<\delta\lesssim \eps^{-1}\rho^{-2}$ be a dyadic number (note that $\delta$ could be big,
depending on $\rho$), and  let $\J$ be  the set of points
which partition the interval $[-1,1]$ into (dyadic) intervals of the same length  $\eps\rho^2\de.$
%A typical point of $\I$ will be denoted by $t^0.$

\smallskip
Similarly, for $i=1,2,$ we choose a finite  equidistant partition $\I_i$  of  width
$\rho(1\wedge\delta)$  of  the interval $I_i$  by  points $y_i^0\in \I_i.$
Note: if $\de>1,$ then $\rho(1\wedge\delta)=\rho,$ and we can choose for $\I_i$ just
 the
 singleton  $\I_i=\{y_i^0\},$ {where $y_i^0$ is the left endpoint of $I_i.$}
In view of \eqref{U1c}, \eqref{U2c} and in analogy with \cite{bmv17}, we then define:
\begin{definr}
For any parameters $x^0_1,t^0_2\in\J,$  $y^0_1\in\I_1$  defined in the previous lines and $y^0_2$ the left
endpoint of $I_2,$ we  define the
sets
\begin{align}
U_1^{x^0_1,y_1^0,\delta} :=\{(x_1,y_1)&:0\le y_1-y_1^0< \rho(1\wedge\delta),\,
0\le x_1-x^0_1+\eps\tfrac{h''(y_1^0)}{2}(y_1-y_1^0)< \eps\rho^2\delta \}, \nonumber\\
\label{whitneybox1}&\\
U_2^{t^0_2,y_1^0,y^0_2,\delta} :=\{(x_2,y_2)&:0\le y_2-y^0_2<\rho, \nonumber\\&\quad0\le x_2-t^0_2+\eps[h'(y_2)-h'(y_1^0)-\tfrac{h''(y_1^0)}{2}(y_2-y_1^0) ]<
\eps\rho^2\delta\},\nonumber%\label{whitneybox2}
\end{align}
and the points
\begin{equation}\label{pointsinU}
z^0_1=(x^0_1,y^0_1), \qquad z^0_2=(x^0_2,y^0_2)
\end{equation}
where \begin{equation*}
x_2^0:=t^0_2-\eps[h'(y^0_2)-h'(y_1^0)-\tfrac{h''(y_1^0)}{2}(y^0_2-y_1^0) ].
\end{equation*}
\end{definr}

Observe that then
$$
z^0_1\in U_1^{x^0_1,y_1^0,\delta}\subset V_1\quad \text{  and   } \quad z^0_2\in
U_2^{t^0_2,y_1^0,y^0_2,\delta}\subset V_2.
$$
Indeed, $z^0_i$ is in some sense the ``lower left'' vertex of $U_i,$ and the horizontal projection of
$U_2^{t^0_2,y_1^0,y^0_2,\delta}$ equals $I_2.$
Moreover, if we define $a^0$ by \eqref{a0defi}, we have that $x_1^0+a^0=t^0_2,$ so that our definitions of the sets
$U_1^{x^0_1,y_1^0,\delta}$ and $U_2^{t^0_2,y_1^0,y^0_2,\delta}$ are very close to the ones for the  sets $U_1$
and
$U_2$ (cf. \eqref{U1c}, \eqref{U2c}) in the previous subsection.
Notice also that we may re-write
\begin{equation}\label{U2alt}
U_2^{t_2^0,y_1^0,y_2^0,\delta}=\{z_2=(x_2,y_2): 0\le\tau_{z_1^0}(z_1^0,z_2)-a^0<\eps\rho^2\delta, \ 0\le y_2-y_2^0<\rho\}.
\end{equation}
%\texttt{This is not in the cubic perturbation version, but here, we refer to it on page 14.}}

 In particular, $U_1^{x^0_1,y_1^0,\delta}$ is  again essentially a paralellepiped of sidelengths $\sim
 \eps\rho^2\de\times \rho(1\wedge \de),$ containing the point $(x^0_1,y_1^0),$  whose
longer side has slope $y_1^0$ with respect to the $y$-axis.
%{\color{magenta}\sl Comment for Stefan and Detlef: the angles are not $\pi/2,$ it is not a rectangle}
 Similarly, if $\de\ll 1,$
then $U_2^{t^0_2,y_1^0,y^0_2\delta}$ is a thin curved box of width $\sim\eps\rho^2 \de$ and
length $\sim\rho,$  contained in a rectangle of dimension $\sim \rho^2\times \rho$ whose axes are
parallel to the coordinate axes (namely the part of a $\rho^2\delta$-neighborhood of a parabola of curvature $\sim\eps$ containing the
point $(x^0_1,y^0_1)$ which lies  within the horizontal strip  $V_2$).   If $\de\gtrsim 1,$ then
$U_2^{t^0_2,y_1^0,y^0_2\delta}$ is essentially  a rectangular  box of dimension $\sim \eps\rho^2\de\times \rho$
lying
in the same horizontal strip.

Note also that we have chosen to use  the parameter $t^0_2$ in place of using $x^0_2$ here, since with this
choice by \eqref{TV1} the identity
\begin{equation}\label{t2mean}
\tau_{z_1^0}(z_1^0,z_2^0)=t^0_2-x_1^0
\end{equation}
holds true, which will become quite useful in the sequel. We next have to relate the parameters $x_1^0,t^0_2,
y_1^0,y_2^0$ in order to give a precise definition of an admissible pair.

\smallskip
 Here, and in the sequel, we shall always assume that the points $z_1^0,z_2^0$ associated to these parameters
 are given by \eqref{pointsinU}.

\smallskip

\begin{definr}
Let us call a pair
 $(U_1^{x_1^0,y_1^0,\delta},U_2^{t_2^0,y_1^0,y_2^0,\delta})$ an {\it admissible
 pair of type 1
 (at scales $\de,\;\rho$ and contained in $V_1\times V_2$),} if the following two conditions hold
 true:
\begin{align}\label{admissible1}
\frac {C_0^2}4\eps\rho^2\de\le
|\tau_{z_1^0}(z_1^0,z_2^0)|&=|t_2^0-x_1^0|<4 \,C_0^2\eps\rho^2\de,\\
\frac{C_0^2}{512}\eps\rho^2(1\vee
\de)\le|\tau_{z_2^0}(z_1^0,z_2^0)|&< 5\,
C_0^2\eps\rho^2(1\vee \de).\label{admissible2}
\end{align}
By $\cP^{\de}$ we shall denote the set of all admissible pairs of
type 1 at scale $\de$ (and  $\rho$, contained in $V_1\times V_2,$), and by $\cP$
the  corresponding union over all dyadic
scales $\de.$
\end{definr}

Observe that, by \eqref{TV2+TV3}, we have
$\tau_{z_2^0}(z_1^0,z_2^0)-\tau_{z_1^0}(z_1^0,z_2^0)\sim\eps(y_2^0-y_1^0)^2.$
In view of   \eqref{admissible1} and \eqref{yseparation} this shows that condition
\eqref{admissible2} is automatically satisfied, unless $\de\sim 1.$

We remark that it would indeed be more appropriate to denote the sets $\cP^{\de}$ by $\cP^{\de}_{V_1\times V_2},$  but we
want to simplify the notation. In all instances in the rest of the paper $\cP^{\de}$ will be  associated  to a
fixed admissible pair of strips $(V_1,V_2),$ so  that our imprecision will not cause any ambiguity.
The next lemma can be proved by closely following the arguments in the proof of the corresponding Lemma 2.1 in \cite{bmv17}:

\begin{lemnr}\label{sizeofdeltas}
If  $(U_1^{x_1^0,y_1^0,\delta},U_2^{t_2^0,y_1^0,y_2^0,\delta})$ is an admissible
pair of type 1, then for
all  $(z_1,z_2)\in (U_1^{x_1^0,y_1^0,\delta},U_2^{t_2^0,y_1^0,y_2^0,\delta})$ ,
$$
|\tau_{z_1}(z_1,z_2)|\sim_{8} C_0^2 \eps\rho^2\de\mbox{   and   }
|\tau_{z_2}(z_1,z_2)|\sim_{1000} C_0^2 \eps\rho^2(1\vee\de).
$$
\end{lemnr}

%\begin{proof}
%Note that
%\begin{eqnarray}\label{tauid}
%&&\tau_{z_1}(z_1,z_2)=x_2-x_1+y_2(y_2-y_1)=t_2^0-x_1^0\\
%&&\hskip0.5cm+ \big[x_2-t_2^0+y_2(y_2-y_1^0)\big]-
%\big[x_1-x_1^0+y_1^0(y_1-y_1^0)\big]+(y_1^0-y_2)(y_1-y_1^0),
%\nonumber
%\end{eqnarray}
%where, by \eqref{yseparation} and our definition of $U_1^{x_1^0,y_1^0,\delta}$ and
%$U_2^{t_2^0,y_1^0,y_2^0,\delta},$ we have
%$|x_2-t_2^0+y_2(y_2-y_1^0)|<\rho^2\de,\  \big|(x_1-x_1^0)+y_1^0(y_1-y_1^0)\big|<\rho^2\de$
%and
%   $|(y_1^0-y_2)(y_1-y_1^0)|<C_0\rho \cdot\rho(1\wedge\delta)\le C_0\rho^2\de.$
%This shows that
%\begin{align}\label{tausize1}
%|\tau_{z_1}(z_1,z_2)-\tau_{z_1^0}(z_1^0,z_2^0)|\le 2C_0
%\rho^2\de,
% \end{align}
% and in particular in combination with \eqref{admissible1} that
% $|\tau_{z_1}(z_1,z_2)|\sim_{8} C_0^2 \rho^2\de,$
%if we choose $C_0$ sufficiently large.
%
% Similarly,  because of \eqref{TV2+TV3'}, we have
%$$
%\tau_{z_2}(z_1,z_2)-\tau_{z_2^0}(z_1^0,z_2^0)=\tau_{z_1}(z_1,z_2)-\tau_{z_1^0}(z_1^0,z_2^0)
%-(y_2-y_1)^2+(y_2^0-y_1^0)^2,
%$$
%where
%\begin{align*}
%|-(y_2-y_1)^2+(y_2^0-y_1^0)^2|&=|(y_2^0-y_2)+(y_1-y_1^0)|\,|(y_2^0-y_1^0)+(y_2-y_1)|\\
%&\le 2\rho \cdot 2 C_0\rho=4 C_0\rho^2.
%\end{align*}
%
%In combination with \eqref{tausize1} this implies
%\begin{align}\label{tausize2}
%|\tau_{z_2}(z_1,z_2)-\tau_{z_2^0}(z_1^0,z_2^0)|\le
%6C_0 \rho^2(1\vee\de).
% \end{align}
%Invoking also  \eqref{admissible2} this implies $|\tau_{z_2}(z_1,z_2)|\sim_{100} C_0^2
%\rho^2(1\vee\de).$
%\end{proof}

\medskip
Up to now we focused on the case $|\tau_{z_1}(z_1,z_2)|\lesssim|\tau_{z_2}(z_1,z_2)|.$
For the symmetric case, corresponding to the situation where $|\tau_{z_1}(z_1,z_2)|\gtrsim|\tau_{z_2}(z_1,z_2)|$, by
interchanging the roles of $z_1$ and $z_2$  we define
accordingly  for any $t^0_1, x^0_2\in\J,$  $y^0_1$ the left endpoint of $I_1$ and $y^0_2\in\I_2$  the sets $\widetilde{U}_1^{t_1^0,y_1^0,y_2^0,\delta}$ and $\widetilde{U}_2^{x_2^0,y_2^0,\delta}$
in analogy with our discussion in \cite{bmv17}, and denote the  corresponding admissible pairs  $(\widetilde{U}_1^{t_1^0,y_1^0,y_2^0,\delta}, \widetilde{U}_2^{x_2^0,y_2^0,\delta})$ as {\it admissible pairs of type 2}. We shall skip  the details.\\
By $\tilde \cP^{\de},$ we shall denote the set of all admissible pairs of
 type 2  at scale $\de$  (and $\rho$, contained in $V_1\times V_2,$), and by $\tilde \cP$
the  corresponding unions over all dyadic
scales $\de.$
\medskip
\medskip

In analogy with Lemma \ref{sizeofdeltas}, we have
\begin{lemnr}\label{sizeofdeltas2}
If $(\tilde U_1,\tilde U_2)=({\tilde U}_1^{t^0_1,y_1^0,y_2^0,\delta},{\tilde U_2}^{x^0_2,y_2^0,\delta})\in \tilde \cP^\de$  is an admissible pair of type 2, then for all $(z_1,z_2)\in(\tilde U_1,\tilde U_2)$ we have
$$
|\tau_{z_1}(z_1,z_2)|\sim_{1000} C_0^2 \eps\rho^2(1\vee\de)\mbox{   and   }
|\tau_{z_2}(z_1,z_2)|\sim_8 C_0^2 \eps\rho^2\de.
$$
\end{lemnr}

\setcounter{equation}{0}
\section{The bilinear estimates}\label{sect:bilin}

\subsection{A prototypical admissible pair in the curved box case and the crucial scaling
transformation}\label{proto}

In this section we shall present a  \lq\lq prototypical"
case where  $U_1$ and $U_2$ will form an admissible pair of type 1 centered at  $z_1^0=0\in U_1$ and $z_2^0\in
U_2,$  with  $\eps\sim1,\rho\sim1$ and $\delta\ll1$, i.e., $|y_1^0-y_2^0|\sim 1,$ and $|\tau_{z_2^0}(z_1^0,z_2^0)|\sim 1$  but
$|\tau_{z_1^0}(z_1^0,z_2^0)|\sim \delta\ll 1.$  This means that we shall be in the curved box case.
%Recall the identity given by \eqref{gammaz}.
 As we  will show in Subsection \ref{scaling transform} in
detail, we can
always reduce to this particular  situation
when the two transversalities  $\tau_{z_2^0}(z_1^0,z_2^0)$ and $\tau_{z_1^0}(z_1^0,z_2^0)$
are of quite different sizes.

Fix  a small number $0<c_0\ll1$ ($c_0=10^{-10}$ will, for instance, work). Assume that
$0<\delta\le1/10,$ and put
\begin{align}
	U_1:=&[0,c_0^2\delta)\times [0,c_0\delta) \label{U1}\\
	U_2:=&\{(x_2,y_2):0\le y_2-b< c_0,0\le x_2+F'(y_2)-a<c_0^2\delta\},\label{U2}
\end{align}
where $|b|\sim_21$, $|a|\sim_4\delta$ and $F$ is a function of cubic type in the sense of \eqref{cubtype}, i.e.,
\begin{eqnarray}\label{cuberrorF}
\begin{cases}  & F(0)=F'(0)=F''(0)=0,  \\
&|F'''(y')|\sim C_3,\\
&|F^{(l)}(y')|\le C_l \quad\text{for all}\   l\ge 4.\\
\end{cases}
\end{eqnarray}
\noindent\bf Remark. \rm Note that in the case $\eps=1$, if we set $C_0=1/{c_0},$  $\rho=c_0,$ then any admissible pair
$(U_1,U_2)=(U_1^{0,0,\delta},U_2^{a,0,b,\delta})$, as in \eqref{whitneybox1}, would satisfy \eqref{U1} and \eqref{U2} with the above conditions on $a$ and $b$ and suitable $F$.
\medskip

Our bilinear result in this prototypical case  is as follows:
\begin{thmnr}[prototypical case]\label{bilinear}
Let $p>5/3,$ and let  $U_1,U_2$ be as in \eqref{U1}, \eqref{U2}. Assume further that  $\phi(x,y)= xy+F(y),$   where $F$ is a real-valued  smooth perturbation function of cubic type, i.e., satisfying estimates \eqref{cuberrorF}, and denote by
$$
\ext_{U_i} f(\xi)=\int_{U_i} f(x,y) e^{-i(\xi_1 x+\xi_2 y+\xi_3\phi(x,y))} \eta(x,y) \, dx dy, \qquad i=1,2,
$$
the  corresponding Fourier extension operators.  Then, if the constants $c_0$ and $\de\ll1$ in \eqref{U1}, \eqref{U2}  are  sufficiently small,
\begin{align}\label{bilinest}
	 \|\ext_{U_1}(f_1),\ext_{U_2}(f_2)\|_p \leq C_p \, \delta^{\frac{7}{2}-\frac{6}{p}}
\|f_1\|_2\|f_2\|_2
 \end{align}
 for every $f_1\in L^2(U_1)$ and every  $f_2\in L^2(U_2),$
where the constant $C_p$  will  only depend on $p$ and the constants $C_l$ in \eqref{cuberrorF}.
\end{thmnr}

As in \cite{bmv17}, it  turns out that  one cannot directly reduce the bilinear
Fourier extension estimates  in \eqref {bilinest}   to Lee's  Theorem 1.1 in \cite{lee05}, since that would not
give us the optimal  dependence on $\delta$. We shall  therefore
have to be more precise about the required transversality conditions. However, once we have  established the correct transversality conditions in Lemma \ref{transvscaled} below (which is the direct analogue of Lemma 2.3 in \cite{bmv17}),    we can indeed apply our  arguments  from  the proof of Theorem 3.3 in \cite{bmv17} also in the  present situation and  arrive at the  desired  bilinear estimates \eqref{bilinest}.
\medskip

The crucial step will again consist in the following scaling: we introduce new coordinates $(\bar x,\bar y)$ be writing $x=\de \bar x, y=\bar y,$ and then re-scale the phase function $\phi$ by putting
$$\phi^s(\bar x,\bar y):=\frac{1}{\de}\phi(\de \bar x, \bar y)=\bar x \bar y+\frac {F(\bar y)}{\de}.$$
 Denote by $U_i^s$ the corresponding re-scaled domains, i.e.,
 \begin{eqnarray*}
U^s_1&=&\{(\bar x_1,\bar y_1): 0\le \bar x_1< c_0^2 ,\ 0\le  \bar y_1< c_0\delta   \},\\
U^s_2&=&\{(\bar x_2,\bar y_2): 0\le  \bar x_2+\frac {F'(\bar y_2)}{\de}-\bar a< c_0^2, \ 0\le  \bar y_2-\bar b< c_0\},
\end{eqnarray*}
where $c_0$ is small and $|\bar a|=|a/\de|\sim 1$ and $\bar b=b\sim 1.$
By $S_i^s, i=1,2,$ we denote the corresponding   scaled surface patches
$$
S_i^s:=\{(\bar x,\bar y,\phi^s(\bar x,\bar y)): (\bar x,\bar y)\in U_i^s\}.
$$

%\begin{proof}
Observe that
\begin{align*}%label{nablaphis}
	\nabla\phi^s(\bar x,\bar y)=(\bar y,\bar x+F'(\bar y)/\de),
\end{align*}
and
\begin{align*}%\label{hessphis}
	H\phi^s(\bar x,\bar y)=
	\left(\begin{array}{cc}
    0  & 1  \\
    1 & F''(\bar y)/\de
\end{array}\right),
\end{align*}
so that in particular
\begin{align}\label{bddgrad}
	|\nabla\phi^s(\bar z)|\lesssim 1
\end{align}
for all $\bar  z\in U^s_1\cup U^s_2$.

Assume next that  $\bar  z_1\in U^s_1$ and $\bar  z_2\in U^s_2.$  Since $|\bar y_1|\le c_0\de,|\bar y_2|\sim 1,$  we see that
 \begin{equation}\label{Fprimes}
 \begin{cases}  |\frac{F'(\bar y_1)}\de |\sim\frac{|F'''(\eta_1)\bar y_1^2|}\de \lesssim C_3c_0^2\frac{\de^2}\de=c_0^2 C_3\de, \quad
\frac{|F''(\bar y_1)|}\de \sim \frac{|F'''(\tilde\eta_1)\bar y_1|}\de \lesssim c_0C_3,\\
 |\frac{F'(\bar y_2)}\de |\sim\frac{|F'''(\eta_2)\bar y_2^2|}\de \sim\frac{ C_3}\de, \quad
\frac{|F''(\bar y_2)|}\de \sim \frac{|F'''(\tilde\eta_2)\bar y_2|}\de \sim \frac{C_3}\de
\end{cases}
\end{equation}
(for suitable choices of intermediate points $\eta_i,\tilde \eta_i$).  Moreover, we then also see that
\begin{equation}\label{graddiffs}
\nabla\phi^s(\bar z_2)-\nabla\phi^s(\bar z_1) =\big(\bar y_2-\bar y_1,\bar x_2+\tfrac {F'(\bar y_2)}{\de}-(\bar x_1+\tfrac {F'(\bar y_1)}{\de})\big)
	=(\bar b, \bar a)+\Landau(c_0).
\end{equation}
	
Following further on the proof of  Lemma 2.3 in \cite{bmv17}, assume that we translate the two patches of surface  $S_1^s$ and $S_2^s$ in such a way  that the two points $\bar z_1$ and $\bar z_2$  coincide after translation, and  assume that the vector
 $\omega=(\omega_1,\omega_2)$ is tangent to the corresponding   intersection curve $\gamma(t)$  at this point. Then \eqref{graddiffs} shows that we may assume without loss of generality that
\begin{align*}%\label{tangent}
	\omega=(-\bar a,\bar b)+\Landau(c_0).
\end{align*}
In combination with \eqref{Fprimes} this implies that
\begin{align*}
	H\phi^s(\bar z_i)\cdot\trans\omega = \left(\begin{array}{cc}
    0  & 1  \\
    1 & F''(\bar y_i)/\de
\end{array}\right)
\left(
\begin{array}{cc}
  -\bar a+ \Landau(c_0)  \\
\bar b+ \Landau(c_0)  \\
\end{array}
\right).
\end{align*}
Thus, if $i=1,$ then by \eqref{Fprimes},
\begin{equation}\label{H1}
H\phi^s(\bar z_1)\cdot\trans\omega =
\left(
\begin{array}{cc}
  \bar b+ \Landau(c_0) \\
-\bar a  + \Landau(c_0) \\
\end{array}
\right)
\quad \text{ and   }\quad  |H\phi^s(\bar z_1)\cdot\trans\omega| \sim 1,
\end{equation}
and if $i=2,$ then
\begin{equation}\label{H2}
H\phi^s(\bar z_2)\cdot\trans\omega =
\left(
\begin{array}{cc}
  \bar b+ \Landau(c_0) \\
-\bar a  + \bar b F''(\bar y_2)/\de+\Landau(c_0)/\de \\
\end{array}
\right)
\quad \text{ and   }\quad  |H\phi^s(\bar z_1)\cdot\trans\omega| \sim 1/\de,
\end{equation}
if $\de\ll1$ is sufficiently small.
\smallskip

Following \cite{bmv17}, the  refined transversalities that we need to control are given by
\begin{equation}\label{transnew}
\Big|TV^s_i(\bar z_1,\bar z_2)\Big|:=\Big|
\frac{\det(\trans (\nabla\phi^s(\bar
z_1)-\nabla\phi^s( \bar z_2)),  H\phi^s( \bar z_i)\cdot\trans\omega)}
{\sqrt{1+|\nabla\phi^s( \bar z_1)|^2}\sqrt{1+|\nabla\phi^s(\bar z_2)|^2}\, |H\phi^s(\bar
z_i)\cdot\trans\omega|}\Big|, \qquad i=1,2.
\end{equation}
	
	But, if $i=1,$ then by \eqref{graddiffs}, \eqref{H2}, \eqref{bddgrad}  and \eqref{Fprimes} we see that
$$
	|\det(\trans (\nabla\phi^s(\bar
z_1)-\nabla\phi^s( \bar z_2)),  H\phi^s( \bar z_1)\cdot\trans\omega)|=\Big|\det\left(
\begin{array}{ccc}
    \bar b+\Landau(c_0)  &  \bar b+\Landau(c_0)  \\
  \bar a+\Landau(c_0)    &  -\bar a+\Landau(c_0) \\
\end{array}
\right)
\Big|\sim 1,
$$
hence $\Big|TV^s_1(\bar z_1,\bar z_2)\Big|\sim1$.

 And, if $i=2,$ then by \eqref{graddiffs}, \eqref{H1}, \eqref{bddgrad}  and \eqref{Fprimes} we have
$$
	|\det(\trans (\nabla\phi^s(\bar
z_1)-\nabla\phi^s( \bar z_2)),  H\phi^s( \bar z_2)\cdot\trans\omega)|=\Big|\det\left(
\begin{array}{ccc}
    \bar b+\Landau(c_0)  &  \bar b+\Landau(c_0)  \\
  \bar a+\Landau(c_0)    &  -\bar a+\bar b F''(\bar y_2)/\de +\Landau(c_0)/\de \\
\end{array}
\right)
\Big|,
$$
hence also $\Big|TV^s_2(\bar z_1,\bar z_2)\Big|\sim (1/\de)/(1/\de)\sim 1,$  provided $\de$ and $c_0$ are
sufficiently small.
\medskip

We have thus proved the following lemma, from which Theorem \ref{bilinear} can easily be derived, as explained before,  by applying  the arguments from \cite{bmv17}:

\begin{lemnr}\label{transvscaled}
The transversalities for the scaled patches of surface  $S_i^s, \,i=1,2,$ satisfy
$$
\Big|TV^s_i(\bar z_1,\bar z_2)\Big|\sim 1, \quad i=1,2.
$$
\end{lemnr}

\color{black}

%\medskip
We should again like to mention that  estimate \eqref{bilinest} could alternatively also be deduced from Candy's Theorem 1.4 in \cite{can17}, after applying the crucial scaling in $x$ that we used in the first step of  our proof.\\

\medskip
\subsection{Reduction to the prototypical case}\label{scaling transform}
Our next goal  will be to establish the following analogues of the bilinear
 Fourier extension  estimates in Theorem 3.1 of \cite{bmv17}:
\begin{thmnr}\label{bilinear2}
Let $p>5/3,$ $q\ge2.$ Then, for every admissible pair  $(U_1,U_2)\in \cP^\de $
at scale $\de,$  the following bilinear estimates hold true:
If $\de>1$ and $\epsilon \de\rho^2\le 1$, then
  \begin{align*}
	 \|\ext_{U_1}(f)\ext_{U_2}(g)\|_p
	\leq C_{p,q}  (\epsilon\delta\rho^3)^{2(1-\frac 1p-\frac 1q)}\|f\|_q\|g\|_q.
 \end{align*}
If  $\de\le 1,$ then
 \begin{align*}
	 \|\ext_{U_1}(f)\ext_{U_2}(g)\|_p
	\leq C_{p,q}\, \,(\epsilon\rho^3)^{2(1-\frac 1p-\frac 1q)} \,  \delta^{5-\frac 3q-\frac 6p}\|f\|_q\|g\|_q.
 \end{align*}
The  constants in these estimates are independent of the given admissible pair, of $\eps, \rho$ and of $\de.$ The same
 estimates are valid for admissible pairs
 $(\tilde U_1,\tilde U_2)\in \tilde\cP^\de $ of type 2.
\end{thmnr}

%\begin{remnr}\label{remarkbilin}
%Recall that for $\de>1$ the sets $U_1$ and $U_2$ are essentially rectangular boxes of
%dimension $\rho^2\de\times \rho,$  and notice that our estimates  for this case do agree
%with the ones given in Proposition 2.1 in \cite{v05} for the case of the saddle.
%\end{remnr}

%We explain in more detail here the reduction of Theorem \ref{bilinear2} to
%Theorem \ref{bilinear}.
%Subsection \ref{bilinarg}  will  then be devoted to the proof of
%Theorem \ref{bilinear}.
Fix $p>5/3$ and $q\ge 2,$  and
assume that
$U_1=U_1^{x_1^0,y_1^0,\delta}$ and $U_2= U_2^{t_2^0,y_1^0,y^0_2,\delta}$ form an admissible pair
of
type 1. We shall only discuss the  case of admissible pairs of type 1;  the type 2 case can be
handled in the same way by symmetry.
%So, let $(U_1,U_2)\in\cP^\de$ be an admissible pair of type 1, where
%$U_1=U_1^{x_1^0,y_1^0,\delta}$ and $U_2= U_2^{x_2^0,y_1^0,\delta}.$

We shall see that  the bilinear estimates associated to the  sets
$U_1,U_2$ can easily be
reduced by means of a suitable
affine-linear transformation to either the  classical bilinear estimate in \cite{lee05},
when $\delta\ge 1/10,$
or
to the estimate for the  special  ``prototype'' situation  given in  Subsection
\ref{proto}, when $\delta\le 1/10.$

We first change  to the coordinates $(x'',y'')$ introduced in \eqref{newcoord}, which allows to reduce to the case where $(z_1^0)''=0$ and
$(z_2^0)''=(x_2^0-x_1^0+\eps \frac{h''(y_1^0)}2 (y_2^0-y_1^0), y_2^0-y_1^0).$ Recall, however, that we need here  to replace our original perturbation  $h(y)$ by  the cubic type perturbation $H(y'')$ (compare \eqref{cuberror}).   In these coordinates, $U_1$ corresponds to the set
\begin{eqnarray*}
U''_1&:=&\{(x''_1,y''_1): 0\le y''_1< \rho(1\wedge\delta),\ 0\le x''_1< \eps\rho^2\de  \},
\end{eqnarray*}
and $U_2$ to the set
\begin{eqnarray*}
	U''_2=\{(x''_2,y''_2): 0\le x''_2+\eps H'(y''_2)-a^0<\eps\rho^2\delta, \ 0\le y''_2-(y''_2)^0<\rho\},
\end{eqnarray*}
where $a^0:=t_2^0-x_1^0$  and $(y''_2)^0:=y_2^0-y_1^0\sim C_0 \rho$ (compare \eqref{U2alt} and \eqref{yseparation}, and note that $\tau_{0}(0,z''_2)= x''_2+\eps H'(y''_2)$ in the coordinates $(x'',y'')$).
Recall also from \eqref{admissible1} that $|a^0|\sim C_0^2\eps \rho^2\de.$

\medskip

This suggests to apply the following  {\bf scaling:}
we change to yet other coordinates $z'=(x',y')$  by writing
\begin{equation}\label{scale1}
y''=\rho  y', \ x''=\eps\rho^2 (1\vee\delta)x'.
\end{equation}
Let us accordingly introduce the function
\begin{equation}
F(y'):=\frac {H(\rho y')}{\rho^3},
\end{equation}
and note that the crucial phase function $x'' y''+\eps H(y'')$  that arose  from $\phi$ in \eqref{newcoordphi} after the change to the coordinates $(x'',y'')$ assumes the following form in the coordinates $(x',y'):$
\begin{equation}\label{phiz'}
x'' y''+\eps H(y'')=\eps \rho^3 (1\vee\delta)\Big(x'y'+\frac{F(y')}{1\vee\delta}\Big)=:\eps \rho^3 (1\vee\delta) \, \phi_\de(x',y').
\end{equation}
Observe that also the function {\it $F$ is  a perturbation  function of cubic type, uniformly  also in $\eps$ and $\rho.$} Indeed, the following holds true:
\begin{eqnarray}\label{cuberrorF2}
\begin{cases}  & F(0)=F'(0)=F''(0)=0,  \\
&|F'''(y')|=|H''(\rho y')|\sim C_3,\\
&|F^{(l)}(y')|=|\rho^{l-3}H^{(l)}(\rho y')| \le C_l \quad\text{for all}\   l\ge 4.\\
\end{cases}
\end{eqnarray}

Thus, altogether we define a change of coordinates $z'=T(z)$ by
\begin{align*}
	x':=&\eps^{-1}(1\vee\delta)^{-1}\rho^{-2}(x-x_1^0+\eps\frac{h''(y_1^0)}{2}(y-y_1^0)),\\
	y':=&\rho^{-1}(y-y_1^0).
\end{align*}
Notice that the following lemma, in the case $\delta\le 1/10,$ corresponds to the prototypical setup up to another harmless scaling $(x',y')=(C_0^{2}x''',C_0y''')$.
\begin{lemnr}\label{U'}
We have
\begin{align}\label{phi-phide}
	\phi(z)=\eps\rho^3(1\vee\de)\phi_\de(Tz)+L(z),
\end{align}
where $L$ is an affine-linear map. Moreover, in these new coordinates, $U_1,U_2$
correspond
to the sets
\begin{eqnarray} \label{U1'U2'}
\begin{split}
U'_1&:=&\{(x'_1,y'_1): 0\le y'_1< 1\wedge\delta,\ 0\le x'_1< 1\wedge\delta  \}=[0,1\wedge \de[^2,\\
U'_2&=&\{(x'_2,y'_2): 0\le x'_2+\frac{F'(y'_2)}{1\vee\delta}-a<1\wedge\delta, \ 0\le y'_2-b<1\},
\end{split}
\end{eqnarray}where $|b|:=|\rho^{-1}(y_2^0-y_1^0)|\sim_2 C_0$  and
$|a|:=|\eps^{-1}\rho^{-2}(1\vee\delta)^{-1}(t_2^0-x_1^0)|\sim_{4} C_0^2\frac{\delta}{1\vee\delta} =
C_0^2(1\wedge\delta).$
Moreover, for Lee's transversality expression $\Gamma^{\phi_\de}$  in \eqref{transs} for
$\phi_\de$ , we have that
\begin{align}\label{transscaled}
|\Gamma^{\phi_\de}_{\tilde z'_1}(z'_1,z'_2)|\sim  C_0^3(1\wedge \de)\quad \text{for all }
\tilde
z'_1\in U'_1, \quad |\Gamma^{\phi_\de}_{\tilde z'_2}(z'_1,z'_2)|\sim  C_0^3 \quad
\text{for
all }
\tilde z'_2\in U'_2,
\end{align}
for every $z'_1\in U'_1$ and every $z'_2\in U'_2.$ Also, for $\delta\ge 1/10,$ the derivatives
of
$\phi_\delta$ can be uniformly
(independently of $\delta$) bounded from above.
\end{lemnr}

The proof, if not clear from our previous discussions, is similar to the proof of Lemma 2.4 in \cite{bmv17}, so we will skip the details.

\bigskip

\noindent {\bf Reduction of Theorem \ref{bilinear2}  to Theorem \ref{bilinear}.}  Consider the scaled sets $U_1',U_2'$  from  Lemma \ref{U'}.

{\bf The case $\de> 1/10$.\footnote{We don't need to distinguish precisely the two cases $\delta>1$ and $\delta\leq1$ from the Theorem, since the desired bounds are comparable for $\delta\sim1$.}} In this case, we see that, $U'_1$ and $U'_2$ are squares  of of small side length $2c_0,$ separated by a distance of size $1,$ and moreover \eqref{transscaled} shows that all relevant transversalities are of size $1.$
Therefore we see that the conditions of Lee's Theorem 1.1 in \cite{lee05} are satisfied for the patches of
 surface  $S'_1$ and $S'_2$ which are the graphs of $\phi_\de$  (defined in  \eqref{phiz'}) over the sets $U'_1$ and
 $U'_2.$  This implies that for these patches of surface, we obtain   uniform bilinear Fourier extension estimates when $p>5/3$ and $q\ge 2,$ of the form
\begin{align*}%\label{bil1'}
 \|\ext_{U'_1}(\tilde f)\,\ext_{U'_2}(\tilde g)\|_p \leq C_{p,q} \|\tilde f\|_q\|\tilde g\|_q,
\end{align*}
with a constant $C_{p,q}$  which is independent of the choice of   $x_1^0,y_1^0, t_2^0, y_2^0,\eps, \rho$ and $\de.$  By scaling back to our original coordinates, we thus arrive at the estimate in the first case of Theorem \ref{bilinear2} (compare  with the scaling argument in Sections 2.5 and 3 of \cite{bmv17}).

\bigskip
{\bf The case $\de\leq 1/10$.} By a harmless scaling %by $C_0^2$ in $x$ and $C_0$ in $y$,
$(x',y')=(C_0^{2}x,C_0y),$
the sets $U'_1$ and $U'_2$ given by \eqref{U1'U2'} transform to

\begin{eqnarray}\label{U1U2'}
\begin{split}
U_1&=\{(x_1,y_1): 0\le x_1<c_0^2\delta ,\ 0\le y_1< c_0\delta\}=[0,c_0^2 \de)\times[0,c_0 \de),\\
U_2&=\{(x_2,y_2): 0\leq x_2+c_0^2F'(\frac{y_2}{c_0})-a< c_0^2\delta,\ 0\leq y_2-b< c_0\},
\end{split}
\end{eqnarray}
where $c_0=C_0^{-1}$ is small and $|a|\sim \de$ and $b\sim 1.$
Recall also that $c_0^3F(\frac{y_2}{c_0})$ satisfies the cubic type estimates \eqref{cuberrorF2}. For the sake of simplicity, let us denote  this perturbation of cubic type again  by $F,$ 
 so that in this case the phase $\phi_\de$, given  by \eqref{phiz'}, can be written as
$$
\phi_\de(x',y')= x'y'+F(y').
$$
\color{black}
This means that we are in the prototypical situation.
The claimed estimates for Case 2 in Theorem \ref{bilinear2} will now follow directly from Theorem \ref{bilinear} for the prototypical case in combination with H\"older's inequality (to pass from $L^2$-norms to $L^q$-norms), if we again scale back to our original coordinates.\qed

%
%\setcounter{equation}{0}
%\section{The Whitney-type decomposition and its overlap}\label{whitn}
%

\setcounter{equation}{0}
\section{The Whitney-decomposition and passage to linear restriction estimates: proof of Theorem
\ref{mainresult}}\label{bilinlin}

%\subsection{The Whitney-type decomposition of $V_1\times V_2$}\label{whitneydecomp}

In order to complete the proof of Theorem \ref{mainresult}, let us finally briefly sketch how to pass from the bilinear estimates in Theorem \ref{bilinear} to the crucial linear estimate in \eqref{extka1}.  Again we shall closely follow our approach in \cite{bmv17} and only indicate the necessary changes.\\
Let $(V_1,V_2)$ be an admissible pair of strips as defined in Subsection \ref{preciseadmissible}.
Recall the definition of admissible pairs  of sets from the same subsection, and that we had also introduced
there the sets  $\cP^\de$  respectively $\tilde \cP^\de$ of  admissible pairs of type
1 respectively   type 2 at scale
$\de,$ and by $\cP$ respectively $\tilde \cP$ we had denoted the  corresponding unions over all dyadic
scales $\de.$
The next lemma is in direct analogy to Lemma 4.1 in \cite{bmv17} and can be proved in a similar fashion.

\begin{lemnr}\label{covering}
The following covering and overlapping properties hold true:
\begin{itemize}
\item[(i)] For fixed dyadic scale $\de,$ the subsets $U_1\times U_2, \, (U_1,U_2)\in
    \cP^\de,$ of $V_1\times V_2\subset Q\times Q$ are pairwise disjoint,  as likewise
    are the subsets $\tilde U_1\times \tilde U_2, \, (\tilde U_1,\tilde U_2)\in
    \tilde\cP^\de.$
\item[(ii)] If $\de$ and $\de'$ are dyadic scales,  and if  $(U_1,U_2)\in \cP^\de$ and
    $(U'_1,U'_2)\in \cP^{\de'},$ then the sets $U_1\times U_2$ and $U'_1\times U'_2$
    can only intersect if   $\de/\de' \sim_{2^7} 1.$ In the latter case, there is only
    bounded overlap. I.e., there is a constant  $M\le2^6$  such that for
    every $(U_1,U_2)\in \cP^\de $  there are at most $M$  pairs   $(U'_1,U'_2)\in
    \cP^{\de'}$  such that
$(U_1\times U_2)\cap ( U'_1\times U'_2)\ne \emptyset,$  and vice versa.
The analogous statements apply to admissible pairs in $\tilde\cP.$
\item[(iii)]  If  $(U_1,U_2)\in \cP^\de$ and  $(\tilde U_1,\tilde U_2)\in
    \tilde\cP^{\de'},$ then $U_1\times U_2$ and $\tilde U_1\times \tilde U_2$ are
    disjoint too, except possibly when both $\de,\de'\ge 1/800$ and
    $\de\sim_{2^{10}}\de'.$ In the latter case, there is only bounded overlap. I.e.,
    there is a constant  $N=\Landau (C_0)$   such that for every $(U_1,U_2)\in \cP^\de
    $  there are at most $N$  pairs   $(\tilde U_1,\tilde U_2)\in \tilde\cP^{\de'}$
    such that
$(U_1\times U_2)\cap (\tilde U_1\times \tilde U_2)\ne \emptyset,$  and vice versa.
\item[(iv)]  The product sets associated to all admissible pairs cover $V_1\times V_2$
    up to a set of measure $0,$  i.e.,
$$V_1\times V_2=\Big(\bigcup\limits_{(U_1,U_2)\in \cP}U_1\times U_2
\Big)\cup\Big(\bigcup\limits_{(\tilde U_1,\tilde U_2)\in \cP}\tilde U_1\times \tilde
U_2\Big)$$
in measure.
\end{itemize}
\end{lemnr}

%To prove Theorem \ref{mainresult}, assume that $r>10/3$ and $1/q'>2/r,$ and put $p:=r/2,$
%so that $p>5/3$, $1/q'>1/p.$ By interpolation with the trivial estimate for
%$r=\infty,q=1,$
%it is enough to prove the result  for $r$ close to $10/3$ and  $q$ close to 5/2, i.e.,
%$p$
%close to $5/3$ and $q$ close to 5/2. Hence, we may  assume that $p<2,$ $p<q<2p.$ Also, we
%can assume that $\supp f\subset\{(x,y)\in Q:  y>0\}.$

To handle the bounded overlap between the sets $U_1\times U_2$  for pairs of admissible sets $(U_1,U_2)\in \cP$ of type 1 in Lemma \ref{covering}, we  define for   $\nu=0,\dots, 9$  the subset  $\cP_\nu:=\bigcup_j \cP^{2^{10j+\nu}}$ of $\cP.$
To these, we associate the subsets
$$
A_\nu:=\bigcup\limits_{(U_1,U_2)\in \cP_\nu}U_1\times U_2, \qquad \nu=0,\dots, 9,
$$
and likewise introduce  the corresponding subsets  $\tilde A_\nu$ associated to admissible pairs of type 2.
Then we may argue as in \cite{bmv17}  to show that it will suffice to prove restriction estimates over these sets $A_\nu,$  respectively $\tilde A_\nu, $ over which we have ``decoupled'' the overlaps. Let us just look at the sets $A_\nu$ in the sequel.
\medskip

To prove Theorem \ref{mainresult}, assume that $r>10/3$ and $1/q'>2/r,$ and put $p:=r/2,$
so that $p>5/3$, $1/q'>1/p.$ By interpolation with the trivial estimate for
$r=\infty,q=1,$
it is enough to prove the result  for $r$ close to $10/3$ and  $q$ close to 5/2, i.e.,
$p$
close to $5/3$ and $q$ close to 5/2. Hence, we may  assume that $p<2,$ $p<q<2p.$ Also, we
can assume that $\supp f\subset\{(x,y)\in Q:  y{\color{red}\geq}0\}.$

\medskip

As in \cite{bmv17}, we easily see that it will suffice to prove the following:
assume a scale $\rho$ is fixed, and that  $V_1\sim V_2$   is an admissible pair of strips at scale $\rho$
(as defined in \eqref{Vi} of Subsection \ref{preciseadmissible}).  Then the following holds true:

\begin{lemnr}\label{V1V2bilin}
 If  $V_1\sim V_2$ form an  admissible pair of ``strips''
 $V_i=V_{j_i,\rho}=[-1,1]\times I_{j_i,\rho}, \, i=1,2,$  at scale $\rho$ within
$Q,$  and if $f\in L^q(V_1)$ and $g\in L^q(V_2),$ then for $5/3<p<2$, $p<q<2p$
%the range of $p$'s and $q$'s described above
 we have
\begin{align}\label{VVbilin}
\|\ext_{V_1} (f) \ext_{V_2} (g)\|_p\lesssim  C_{p,q} \,\rho^{2(1-1/p-1/q)} \|f\|_q\,
\|g\|_q \  \text{for all} \ f\in L^q(V_1),  g\in L^q(V_2).
\end{align}
\end{lemnr}

We remark that, eventually, we shall choose $f=g$,  but for the arguments to follow it is
helpful to distinguish between  $f$ and $g$.

To prove this lemma, observe first that by  means of an affine linear transformation we may ``move the strips $V_1, V_2$  vertically'' so that $j_1=0,$ which means  that $V_1$ contains the origin and, by \eqref{admissibleV}, $j_2\sim C_0.$  This we shall assume throughout the proof.
 \smallskip

As mentioned before, it will suffice to estimate
$E((f\otimes g)\chi_{A_\nu})$  in place of $\ext_{V_1} (f) \ext_{V_2} (g),$ and the same  arguments as in \cite{bmv17} then show
 that we may decompose
$$
(f\otimes g)\chi_{A_\nu}=\sum_\delta \sum_{i,i',j} f_{i,j}^\delta\otimes g_{i',j}^\delta,
$$
where $$f_{i,j}^\delta=f\chi_{U_1^{i\eps\rho^2\de,j\rho(1\wedge\de),\delta}},\ \  g_{i',j}^\delta=g\chi_{
U_2^{i'\eps\rho^2\de,j\rho(1\wedge\de),j_2\rho,\delta}},$$ and where  each
$\big(U_1^{i\eps\rho^2\de,j\rho(1\wedge\de),\delta}, U_2^{i'\eps\rho^2\de,j\rho(1\wedge\de),j_2\rho,\delta}\big)$ forms an
admissible pair,  i.e., \eqref{admissible1}, \eqref{admissible2} are satisfied. This
means
in particular that  $ |i-i'|\sim C_0^2.$ The summation  in $\de$ is here meant as
summation
over all dyadic $\de$ such that $\de\lesssim (\eps\rho^2)^{-1}.$

We may and shall also assume that
$f$ and $g$ are supported on the set $\{y\geq0\}.$ Then
\begin{align}\label{EfgA}
E((f\otimes g)\chi_{A_\nu})
	= \sum_{\delta\gtrsim1}\sum_{i,i'} \widehat{f_{i}^\delta
d\sigma}\widehat{g_{i'}^\delta d\sigma}+ \sum_{\delta\ll 1}\sum_{i,i',j} \widehat{f_{i,j}^\delta
d\sigma}\widehat{g_{i',j}^\delta d\sigma}.
\end{align}
The first sum  can be treated by more classical  arguments (compare, e.g., \cite{lee05} or  \cite{v05}), which in view of the first  estimate in Proposition  \ref{bilinear2} then leads to
a  bound for the contribution of that sum to $\|\ext_{V_1} (f) \ext_{V_2} (g)\|_p$ in \eqref{VVbilin} of the order
$$
\sum_{1\lesssim \de\lesssim (\eps\rho^2)^{-1}}C_{p,q}\, \,(\de \epsilon\rho^3)^{2(1-\frac 1p-\frac 1q)} \|f\|_q\|g\|_q
\lesssim\rho^{2(1-\frac 1p-\frac 1q)} \|f\|_q\|g\|_q,
$$
as required. We  leave the details  to the interested reader. Note that for this first  sum, there is no gain when $\eps>0$ is getting small (which is to be expected), in contrast  to what will happen  for the second sum.

\medskip
We  shall now  concentrate on the second  sum in \eqref{EfgA} where $\delta\ll 1.$  Here, the
 admissibility conditions reduce to $|i-i'|\sim C_0^2,$ $j_2\sim C_0.$

We fix $\delta,$ and simplify notation by writing  $f_{i,j}:=f_{i,j}^\delta,$
 $g_{i,j}:=g_{i,j}^\delta,$ and $U_{1,i,j}:=U_1^{i\eps\rho^2\de,j\rho(1\wedge\de),\delta}$, $ U_{2,i',j}:=
 U_2^{i'\eps\rho^2\de,j\rho(1\wedge\de),j_2\rho,\delta}$.

 As a first step in proving estimate \eqref{VVbilin}, we exploit some almost orthogonality  with respect to the $x$-coordinate, following a classical approach (compare, e.g., \cite{MVV1}, \cite{MVV2}).

 \begin{lemnr}\label{lpsquare} For $1\le p\le 2,$  we have
\begin{align*}%\label{psquarefunc}
	\big\|\sum_{i,|i-i'|\sim C_0^2,\, j}\, \widehat{f_{i,j}d\sigma}\,
\widehat{g_{i',j}d\sigma}\big\|_{p}^p
	\lesssim \sum_{N=0}^{(\eps\rho^2)^{-1}}\Big\|\sum_{i\in[N\de^{-1},(N+1)\de^{-1}]\, ,\atop  |i-i'|\sim C_0^2,\,j}  \widehat{f_{i,j}d\sigma}\,\widehat{g_{i',j}d\sigma}\Big\|_{p}^p.
\end{align*}
\end{lemnr}

\begin{proof}
Assume that $ i\in[N\de^{-1},(N+1)\de^{-1}],$ and that $z_1=(x_1,y_1)\in U_{1,i,j}$ and
$z_2=(x_2,y_2)\in U_{2,i',j},$ where $|i-i'|\sim C_0^2,$ which means  that $(U_{1,i,j}, U_{2,i',j})\in \cP^\de$ is an admissible pair. Then, in a similar way as in the proof of the corresponding lemma in \cite{bmv17}, by means of Taylor expansions (where we only need to make use of the estimates for third derivatives of $h$) one sees that
$
|x_2-x_1|\lesssim\,CC_0^2\eps\rho^2.
$
This implies that   $x_1+x_2=2N\eps \rho^2+\Landau(\eps\rho^2),$ where the constant in the error term is of order $C_0^2,$  hence
$$
U_{1,i,j}+U_{2,i',j}\subset [2N\eps\rho^2-C\, C_0^2\eps\rho^2,2N\eps\rho^2+C\,C_0^2\eps\rho^2]\times [0,2C_0\rho].
$$
These statements become even more lucid if we first apply the scaling $y=\rho  y',  x= \eps\rho^2 x',$ that we had already  introduced in \eqref{scale1}, for then we may assume that in our definition of the sets $U_{1,i,j}, U_{2,i',j}$ we have $\eps=1$ and $\rho=1$. We also remark that the constant  $C$ will depend here only on the constant $C_3$ which controls third derivatives of $h$ in \eqref{cubtype}.

Notice that the family of intervals $\big\{[2N\eps\rho^2-C\, C_0^2\eps\rho^2,2N\eps\rho^2+C\,C_0^2\eps\rho^2]\big\}_{N=0}^{(\eps\rho^2)^{-1}}$ is almost pairwise disjoint. Therefore we may argue as in the proof of Lemma 6.1 in  \cite{TVV} in order to derive the desired estimate.
\end{proof}

\smallskip

%For our next step, recall again that if we apply the scaling $y=\rho  y',  x= \eps\rho^2 x'$ from  \eqref{scale1},  then according to  \eqref{U1'U2'} in the new coordinates $z'=(x',y')$ the sets $U_{1,i,j}, U_{2,i',j}$ are given by
%\begin{eqnarray*} %\label{U1'U2'}
%U'_1&:=&\{(x'_1,y'_1): 0\le y'_1< \delta,\ 0\le x'_1< \delta  \}=[0, \de[^2,\\
%U'_2&=&\{(x'_2,y'_2): 0\le x'_2+F'(y'_2)-a<\delta, \ 0\le y'_2-b<1\},
%\end{eqnarray*}
%where $|a|\sim C_0^2\de$ and $b\sim C_0.$ Thus  $U_1'$ is a square of dimension $\de\times\de,$ and $U'_2$  essentially a thin curved box of width $\de$ and length  $1.$  We shall therefore further decompose $U'_2$ into squares
% of  size $\de\times\de.$

% Accordingly, in our original coordinates $z=(x,y),$

We proceed in analogy with \cite{bmv17}: $U_{1,i,j}$ is a  rectangular box, now of dimension
$\eps\rho^2\de\times \rho\de,$ and we shall further decompose the curved box $U_{2,i',j}$ into essentially  rectangular boxes of the same dimensions $\eps \rho^2\de\times \rho\de,$ by decomposing them in the $y$-coordinate into  $\Landau (1/\de)$ intervals of length $\rho\de.$ I.e., we shall put
 $$
 U^k_{2,i',j}:=\{(x,y)\in U_{2,i',j}: 0\le y-k\rho\de< \rho\de\}.
 $$
  Then
  \begin{align*}%\label{U2decomp}
 U_{2,i',j}=\overset{\cdot}{\bigcup\limits_k} \,U^k_{2,i',j},
 \end{align*}
 where  the union is over a set of $\Landau (1/\de)$ indices $k.$  Accordingly, we
 decompose
 $g_{i',j}=\sum_k g_{i',j}^k,$ where
 $g_{i',j}^k:=g\chi_{U^k_{2,i',j}}.$
Then we have the following uniform square function estimate:
\begin{lemnr}\label{squaref} For $1<p\le 2$ there exists a constant $C_p>0$ such that for every  $N=0,\dots,(\eps\rho^2)^{-1}$ we have
\begin{align}\label{squarefunc}
\Big\|\sum_{i\in[N\de^{-1},(N+1)\de^{-1}], \atop  |i-i'|\sim C_0^2,\,j}  \widehat{f_{i,j}d\sigma}\,
\widehat{g_{i',j}d\sigma}\Big\|_{p}
	\le C_p \Big\|\Big(\sum_{i\in[N\de^{-1},(N+1)\de^{-1}],\atop |i-i'|\sim C_0^2,\,j\, ,k} |\widehat{f_{i,j}d\sigma}\,
\widehat{g_{i',j}^kd\sigma}|^2\Big)^{1/2}\Big\|_{p}.
\end{align}
\end{lemnr}
\noindent {\it Proof of Lemma \ref{squaref}:} Notice first that a translation in  $x$ by $N\rho^2$ allows to reduce to the case $N=0,$  which we shall thus assume. Then the relevant sets $U_{1,i,j}$ and $U_{2,i',j}$  will all have their $x$-coordinates in the interval $[0,\eps\rho^2].$

 For $i,\,i',\,j,\,k$ as above, set
$S_{1,i,j}:=\{(\xi,\phi(\xi)): \xi\in U_{1,i,j}\},$
$S_{2,i',j}^k:=\{(\xi,\phi(\xi)):\xi\in
U_{2,i',j}^k\},$
and denote by ${(x',y')=}D_{\eps,\rho}(x,y):=(\eps\rho^2 x,\rho y)$  the scaling transformation which changes  coordinates from  $z=(x,y)$ to $z'=(x',y').$
The key to the square function estimate  \eqref{squarefunc} is the following almost orthogonality lemma:

\begin{lemnr}\label{ao}
Assume $N=0,$ and denote by $\tilde D_{\eps,\rho}, \rho >0,$  the scaling transformation on the ambient space $\Bbb R^3$ which is given by $\tilde D_{\eps,\rho}(x,y,w):=(\eps\rho^2 x, \rho y, \eps\rho^3 w).$  Then there is a family of cubes $\{Q_{i,i',j}^k\}_{i\in[0,\de^{-1}], |i-i'|\sim C_0^2\,,j\, ,k}$  in $\mathbb R^3$ with bounded overlap, whose  sides  are parallel to the coordinate axes  and of  length  $\sim\delta,$     such that
$S_{1,i,j}+S_{2,i',j}^k\subset
\tilde D_{\eps,\rho}( Q_{i,i',j}^k).$
 \end{lemnr}
We remark that the amount of the overlap is in fact  entirely controlled by the size of the constant $C_3$ in \eqref{cubtype}   (and on our choice of $C_0$), but not on the constants $C_l$ for $l\ge 4$ in \eqref{cubtype}.

\medskip
\noindent {\it Proof of Lemma \ref{ao}:} Note first that by our assumptions we have  $V_1,V_2\subset [0,1]\times [0,2C_0\rho].$
Since
$$
\tilde D_{\eps,\rho}^{-1}\big(D_{\eps,\rho}(z'), \phi( D_{\eps,\rho}(z'))\big)=(x',y', x'y'+F(y'))
$$
(compare Subsection \ref{scaling transform}),
we may apply this scaling in order to  reduce our considerations  to the  case where $\eps=\rho=1,$  if we  replace the perturbation term $h$ by the function  $F$ which, according to \eqref{cuberrorF2}, shares the same type of estimates as $h.$  Notice also that, after  scaling,   the  sets  corresponding to $V_1,V_2$ in the new  coordinates
then satisfy  $V_1,V_2\subset [0,(\eps\rho^2)^{-1}]\times [0,2C_0].$

Therefore, from now on we shall work under these assumptions, denoting the new coordinates again by  $(x,y)$ in place  of
$(x',y'),$ in  order to defray the notation.

\smallskip

Notice also that  if $i\in[0,\de^{-1}], |i-i'|\sim C_0^2,$ then  the corresponding patches of surface $S_{1,i,j}$ and $S_{2,i',j}^k$   are contained in boxes of side length, say,  $2\delta,$ and sides parallel to the axes, whose projections to the $x$-axis lie within the unit interval $[0,1].$  Therefore we can choose for
$Q_{i,i',j}^k$  a square  of side length $4\delta,$  with  sides parallel to the axes,
with the property that $S_{1,i,j}+S_{2,i',j}^k\subset Q_{i,i',j}^k.$
We shall prove that the overlap is bounded, with a bound depending only on $C_0$ and the constant  $C_3$ in \eqref{cubtype}.

Note that, if $(x_1,y_1)\in U_{1,i,j}$ and $(x_2,y_2)\in  U_{2,i',j}^k$ with $|i-i'|\sim
C^2_0,$ then, by Lemma \ref{sizeofdeltas} we have
\begin{align*}%\label{odelta}
 \big| x_2-x_1+F'(y_2)-F'(y_1)-\tfrac 12 F''(y_1)(y_2-y_1)\big|\sim C_0^2 \delta.
\end{align*}

It suffices to prove the following: if $(x_1,y_1),(x_2,y_2)$ and $(x_1',y_1'),(x_2',y_2')$
are so that each  coordinate of these points is bounded by a large multiple of $C_0,$ the $y$-coordinates are positive and  satisfy $y_2-y_1\gtrsim C_0, $ $y_2'-y_1'\gtrsim C_0 $ (by the $y$-separation \eqref{yseparation}), and
\begin{eqnarray*}
x_2-x_1+F'(y_2)-F'(y_1)-\tfrac 12 F''(y_1)(y_2-y_1)&\sim&C_0^2\de,\\
x'_2-x'_1+F'(y'_2)-F'(y'_1)-\tfrac 12 F''(y'_1)(y'_2-y'_1)&\sim&C_0^2\de,\\
x_1+x_2&=&x_1'+x_2'+\Landau(\delta),\\
y_1+y_2&=&y_1'+y_2'+\Landau(\delta),\\
x_1y_1+F(y_1)+x_2y_2+F(y_2)&=&x_1'y_1'+F(y'_1)+x_2'y_2'+F(y'_2)
+\Landau(\delta),
\end{eqnarray*}
then
 \begin{align}\label{overlap2}
x_1'=x_1+\Landau(\delta), \ y_1'=y_1+\Landau(\delta),\ x_2'=x_2+\Landau(\delta),\
y_2'=y_2+\Landau(\delta).
\end{align}
To prove this, set
$$
a:=x_1+x_2, \quad b:=y_1+y_2,\qquad a':=x'_1+x'_2, \quad b':=y'_1+y'_2,
$$
and
$$
t_1:=x_1y_1+F(y_1),\qquad t_2:=x_2y_2+F(y_2).
$$
The analogous quantities defined by $(x_1',y_1'),(x_2',y_2')$ are denoted by $t'_1$ and $t'_2.$
Notice that by our assumptions,  $a$ and $b$ only vary of order $\Landau (\de)$ if we
replace $(x_1,y_1),(x_2,y_2)$ by $(x_1',y_1'),(x_2',y_2').$
Then,
$$
t_1+t_2=2x_1y_1-bx_1-ay_1+ab+F(y_1)+F(b-y_1).
$$

We next choose $c$ with $|c|\sim C^2_0,$ such that $x_2-x_1+F'(y_2)-F'(y_1)-\tfrac 12 F''(y_1)(y_2-y_1)=c\delta.$ Then we may
re-write
$$
x_1=\big (a-c\delta+F'(b-y_1)-F'(y_1)-\tfrac 12 F''(y_1)(b-2y_1)\big)/2,
$$
which implies that
\begin{eqnarray*}
t_1+t_2&=&\big(y_1-\frac b 2\big)\Big( a-c\de +F'(b-y_1)-F'(y_1)- F''(y_1)\Big(\frac b 2-y_1\Big)\Big)\\
&&\hskip1cm -ay_1+ab +F(y_1)+F(b-y_1)\\
&=&ab/2+\Landau(\delta)+\psi(y_1),\\
 \end{eqnarray*}
 where we have set
 $$
 \psi(y):=\big(y-\frac b 2\big)[F'(b-y)-F'(y)+ \Big(y-\frac b 2\Big)F''(y)]+F(y)+F(b-y).
 $$
  We compute that the derivative of $\psi$ is given by
\begin{eqnarray}\nonumber
\psi'(y)&=&\big(y-\frac b 2\big)[F''(y)-F''(b-y)+ (y-\frac b 2)F'''(y)]\\
&=&\big(y-\frac b 2\big)^2[2F'''(\eta)+F'''(y)],\label{psideriv}
\end{eqnarray}
where $\eta$ is some intermediate point between $y$ and $b-y.$

 Similarly, $t'_1+t'_2=a'b'/2+\Landau(\delta)+\psi(y'_1).$
Since  $a=a'+\Landau(\delta), b=b'+\Landau(\delta),$ hence $ab=a'b'+\Landau(\delta).$ By our assumption, $t_1+t_2=t_1'+t_2'+\Landau(\delta),$  we conclude that
\begin{equation}\label{psiys}
\psi(y_1)=\psi(y'_1)+ \Landau(\delta).
\end{equation}
Here, the implicit constant in $\Landau(\delta)$ depends so far only on $C_0.$ 
But,  because of the $y$-separation \eqref{yseparation}, we have
$|y_2-y_1|\gtrsim C_0 ,$ and since $b=y_2+y_1,$ we see that   $|y_1-b/2|\sim C_0.$ Moreover, since $|F'''|\sim C_3,$ so that $F'''$ in particular does not change sign, we deduce from \eqref{psideriv} that for all relevant $y$'s we have
$$
|\psi'(y)|\sim C_3 |y-b/2|^2\sim C_3 C_0^2\gg 1,
$$
if we choose $C_0$ sufficiently large.

In combination with \eqref{psiys} this shows that we must have $y_1'=y_1+\Landau(\delta),$  where the implicit constant in $\Landau(\delta)$ depends only on $C_3$ and $C_0,$  hence also
$y_2'=y_2+\Landau(\delta),$ and then our first three assumptions imply also the remaining
assertions in \eqref{overlap2}.

This finishes the proof of the almost orthogonality Lemma \ref{ao}.
\hfill $\Box$
\bigskip

By means of the preceding lemmas and Rubio de Francia's estimate \cite{rdf} (see also \cite{c67}, \cite{co81}) we can now argue in almost exactly the same way as in \cite{bmv17} in order to estimate the contribution of the second sum $\sum_{\delta\ll 1}\sum_{i,i',j} \widehat{f_{i,j}^\delta
d\sigma}\widehat{g_{i',j}^\delta d\sigma}$ in \eqref{EfgA} to $\|\ext_{V_1} (f) \ext_{V_2} (g)\|_p$ in \eqref{VVbilin}. In this way, we  see that it is of  the order
\begin{eqnarray*}
&&\sum_{\delta\ll 1}C_{p,q}\, \ \eps^{2(1-\frac 1p-\frac 1q)}\delta^{5-2/q-7/p} \rho^{6(1-1/p-1/q)}\|f\|_q\|g\|_q\\
&&\lesssim C_{p,q}\, \ \eps^{2(1-\frac 1p-\frac 1q)} \rho^{6(1-1/p-1/q)}\|f\|_q\|g\|_q.
\end{eqnarray*}
This estimate is even stronger than the required estimate in \eqref{VVbilin}. Notice that the additional factor $\eps^{2(1-\frac 1p-\frac 1q)}$ appears here, due to the estimate in Theorem \ref{bilinear2} for Case 2,  which was not present in \cite{bmv17} (where we had $\eps=1$). Also, the power of $\rho$ is better than needed, but these gains do not help
for the total estimate of $\|\ext_{V_1} (f) \ext_{V_2} (g)\|_p,$ because of the presence of first sum in \eqref{EfgA}, in which $\delta\gtrsim 1$. 
We  leave the details  to the interested reader.
\smallskip

This completes the proof of Lemma \ref{V1V2bilin}.
\hfill $\Box$
\medskip

By means of Lemma \ref{V1V2bilin}, we may  finally argue as in the last part of the proof of Theorem 1.1 in \cite{bmv17} in order to sum the contributions by all admissible  pairs of ``horizontal strips''  $V_1\sim V_2$ and arrive at  the estimate \eqref{extka1}, thus completing the proof of Theorem \ref{mainresult}.
\hfill $\Box$
\medskip

%%%%%%%%%%%%%%%%%%%%%%%%%%%%%%%%%%%%%%%%%%%%%%%%%%%%%%%%%%%%%%%%%%%%%%%%%%%%%%
%%%%%%%%%%% References
%%%%%%%%%%%%%%%%%%%%%%%%%%%%%%%%%%%%%%%%%%%%%%%%%%%%%%%%%%%%%%%%%%%%%%%%%%%%%%

\thispagestyle{empty}
%\newpage

\renewcommand{\refname}{References}

\end{document}